\documentclass[twosided]{amsart}
  \usepackage{amsmath,amssymb,amsthm,amsfonts}
  \usepackage{hyperref}
  \usepackage{mathtools,bm}
  \usepackage{latexsym}
  \usepackage{mathrsfs}
  \usepackage{pstricks}
   \usepackage{tikz-cd}
  \usepackage{epsfig}
  \usepackage{verbatim}
  \usepackage{float}
  \usepackage{booktabs}
  \usepackage{enumitem}
  \usepackage{textcomp}
  \usepackage{tikz}
  \usetikzlibrary{matrix,arrows,decorations.pathmorphing}

\providecommand{\U}[1]{\protect\rule{.1in}{.1in}}

  \newcommand{\wt}{\widetilde}
  \newcommand{\gdeg}{G\text{\rm -deg}}

  \newcommand{\vp}{\varphi}
  \newcommand{\ve}{\varepsilon}

  \DeclareMathOperator{\id}{Id}
  \def\bn{\mathbb N}
\def\br{\mathbb R}
\def\bz{\mathbb Z}
\def\noi{\noindent}
  
  \newcommand{\vs}{\vskip .3cm}

  \newcommand\numberthis{\addtocounter{equation}{1}\tag{\theequation}}

  \newcommand\cU{\ensuremath{\mathcal U}}
    
  \newcommand\cV{\ensuremath{\mathcal V}}
  \newcommand\cW{\ensuremath{\mathcal W}}

  \newcommand\bbR{\ensuremath{\mathbb R}}

  \newcommand\bfV{\ensuremath{\mathbf V}}

  \newcommand\scrA{\ensuremath{\mathscr A}}

  \newcommand\scrE{\ensuremath{\mathscr E}}

  \newrgbcolor{violet}{.6 .1 .8}
  \newrgbcolor{lightyellow}{1 1 .8}
  \newrgbcolor{lightblue}{.80 1 1}
  \newrgbcolor{mygreen}{0 .66 .05}
  \definecolor{mygreen}{rgb}{0,.66,.05}
  \definecolor{lightyellow}{rgb}{1,1,.80}
  \newrgbcolor{orange}{1 .6 0}
  \newrgbcolor{GreenYellow}{.85 1 .31}
  \newrgbcolor{Yellow}{1  1  0}
  \newrgbcolor{Goldenrod}{1  .90  .16}
  \newrgbcolor{Dandelion}{1  .71  .16}
  \newrgbcolor{Apricot}{1  .68  .48}
  \newrgbcolor{Peach}{1  .50  .30}
  \newrgbcolor{Melon}{1  .54  .50}
  \newrgbcolor{YellowOrange}{1  .58  0}
  \newrgbcolor{Orange}{1  .39  .13}
  \newrgbcolor{BurntOrange}{1  .49  0}
  \newrgbcolor{Bittersweet}{1.  .4300  .24}
  \newrgbcolor{RedOrange}{1  .23  .13}
  \newrgbcolor{Mahogany}{1.  .4475  .4345}
  \newrgbcolor{Maroon}{1.  .4084  .5376}
  \newrgbcolor{BrickRed}{1.  .3592  .3232}
  \newrgbcolor{Red}{1  0  0}
  \newrgbcolor{OrangeRed}{1  0  .50}
  \newrgbcolor{RubineRed}{1  0  .87}
  \newrgbcolor{WildStrawberry}{1  .04  .61}
  \newrgbcolor{CarnationPink}{1  .37  1}
  \newrgbcolor{Salmon}{1  .47  .62}
  \newrgbcolor{Magenta}{1  0  1}
  \newrgbcolor{VioletRed}{1  .19  1}
  \newrgbcolor{Rhodamine}{1  .18  1}
  \newrgbcolor{Mulberry}{.6668  .1180  1.}
  \newrgbcolor{RedViolet}{.9538  .4060  1.}
  \newrgbcolor{Fuchsia}{.5676  .1628  1.}
  \newrgbcolor{Lavender}{1  .52  1}
  \newrgbcolor{Thistle}{.88  .41  1}
  \newrgbcolor{Orchid}{.68  .36  1}
  \newrgbcolor{DarkOrchid}{.60  .20  .80}
  \newrgbcolor{Purple}{.55  .14  1}
  \newrgbcolor{Plum}{.50  0  1}
  \newrgbcolor{Violet}{.98 .15 .95}
  \newrgbcolor{RoyalPurple}{.25  .10  1}
  \newrgbcolor{BlueViolet}{.84  .38  .98}
  \newrgbcolor{Periwinkle}{.43  .45  1}
  \newrgbcolor{CadetBlue}{.38  .43  .77}
  \newrgbcolor{CornflowerBlue}{.35  .87  1}
  \newrgbcolor{MidnightBlue}{.4414  .9259  1.}
  \newrgbcolor{NavyBlue}{.06  .46  1}
  \newrgbcolor{RoyalBlue}{0  .50  1}
  \newrgbcolor{Blue}{0  0  1}
  \newrgbcolor{Cerulean}{.06  .89  1}
  \newrgbcolor{Cyan}{0  1  1}
  \newrgbcolor{ProcessBlue}{.04  1  1}
  \newrgbcolor{SkyBlue}{.38  1  .88}
  \newrgbcolor{Turquoise}{.15  1  .80}
  \newrgbcolor{TealBlue}{.1572  1.  .6668}
  \newrgbcolor{Aquamarine}{.18  1  .70}
  \newrgbcolor{BlueGreen}{.15  1  .67}
  \newrgbcolor{Emerald}{0  1  .50}
  \newrgbcolor{JungleGreen}{.01  1  .48}
  \newrgbcolor{SeaGreen}{.31  1  .50}
  \newrgbcolor{Green}{0  1  0}
  \newrgbcolor{ForestGreen}{.1992  1.  .2256}
  \newrgbcolor{PineGreen}{.3100  1.  .5575}
  \newrgbcolor{LimeGreen}{.50  1  0}
  \newrgbcolor{YellowGreen}{.56  1  .26}
  \newrgbcolor{SpringGreen}{.74  1  .24}
  \newrgbcolor{OliveGreen}{.6160  1.  .4300}
  \newrgbcolor{RawSienna}{.53  .28  .16}
  \newrgbcolor{Sepia}{1.  .7510  .70}
  \newrgbcolor{Brown}{.41  .25  .18}
  \newrgbcolor{TAN}{.86  .58  .44}
  \newrgbcolor{Gray}{1.  1.  1.}
  \newrgbcolor{Black}{1  1  1}
  \newrgbcolor{White}{1  1  1}

\usepackage{listings}
  \lstdefinelanguage{GAP}{
    basicstyle=\ttfamily,
    keywords={true, false, function, return, fail, if, in, while, do, od, else, elif, fi, break, continue},
    keywordstyle=\color{blue}\bfseries,
    otherkeywords={
      >, <, ==
    },
    identifierstyle=\color{black},
    sensitive=True,
    comment=[l]{\#},
    commentstyle=\color{cyan},
    stringstyle=\color{red},
    morestring=[b]',
    morestring=[b]"
  }

 \theoremstyle{plain}
 \newtheorem{theorem}{Theorem}[section]
 \newtheorem{proposition}[theorem]{Proposition}
 \newtheorem{lemma}[theorem]{Lemma}
 \newtheorem{corollary}[theorem]{Corollary}
 
 \theoremstyle{definition}

 \newtheorem{remark}[theorem]{Remark}

\begin{document}
  \title[Symmetries of Periodic Solutions to Non-Autonomous   Delayed  Systems]{Existence and Spatio-Temporal Patterns of  Periodic Solutions to  Second Order Non-Autonomous  Equivariant Delayed  Systems}
  \author[Z.Balanov --- W. Krawscewicz --- N. Hirano ---X. Ye]{Zalman Balanov --- Wies{\l}aw Krawcewicz --- Norimichi Hirano --- Xiaoli Ye$^*$}
  
  \address
{\textsc{Zalman Balanov}\\
Department of Mathematics\\ Xiangnan University\\
Chenzhou, Hunan 423000, China, and\\
Department of Mathematical Sciences\\
University of Texas at Dallas\\
Richardson, TX 75080, USA}

\email{balanov@utallas.edu}

\address
{\textsc{Wieslaw Krawcewicz}\\
Applied Mathematics Center at Guangzhou University\\
Guangzhou 510006, China, and\\
Department of Mathematical Sciences\\
University of Texas at Dallas\\
Richardson, TX 75080, USA.}

\email{wieslaw@utallas.edu}

\address
{\textsc{Norimichi Hirano}\\
Graduate School of Environment and Information Sciences, \\
Yokohama National University, Yokohama, Japan. }

\email{hira0918@gmail.com}  

\address
{\textsc{Xiaoli Ye}\\
Department of Mathematical Sciences\\
University of Texas at Dallas\\
Richardson, TX 75080, USA}

\email{xiaoli.ye@utdallas.edu}
  \newcommand{\support}{The second author was supported by the National Natural
  	Science Foundation of China (Grant No.~11871171).}
  

  
  \thanks{$^*$ Corresponding author.}
  \thanks{The first author acknowledges the support from Xiangnan University. The second author acknowledges the support from National Science Foundation of China through the grant no. 11871171, and the support from Guangzhou University. }
  
 \maketitle  

\begin{abstract} 

Existence and spatio-temporal symmetric patterns of periodic solutions to second order reversible equivariant non-autonomous periodic systems with multiple delays are studied under the Hartman-Nagumo growth conditions. The method is based on    
using the Brouwer $D_1 \times \mathbb Z_2\times  \Gamma$-equivariant degree theory, where
$D_1$ is related to the reversing symmetry, $\mathbb Z_2$ is related to the oddness of the right-hand-side and $\Gamma$ reflects the
symmetric character of the coupling in the corresponding network. 
Abstract results are supported by a concrete example with $\Gamma = D_n$ -- the dihedral group of order $2n$.

\end{abstract} 




\section{Introduction}

\subsection{Subject and goal} For a long time, the classical forced pendulum equation $\ddot{y} + a \sin(y) = b\sin(t)$ together with its multiple generalizations stand as an important subject, where modern methods and techniques of Nonlinear Analysis are tested (see, for example, \cite{Mawhin2,Mawhin1}).  Among them, the variational methods (Lusternik-Schnirelmann Theory, Morse Theory) and  topological ones (based on different variants of degree theory) have taken a firm position. For the systematic variational treatment of second order Hamiltonian ODEs, we refer  to the monographs \cite{MawhinWillem,Rabinowitz} and references therein.  In \cite{LiHe} (see also \cite{GuoYu}), the variational approach was applied for the first time to autonomous second order delay differential equations (in short, DDEs); for non-autonomous second order periodic DDEs (in short, NASODDEs), we refer to \cite{MengZhang,WuWuZhou,YuXiao} and references therein. Also,
the case of multiple commensurate periods (i.e., the delays being integer multiples of a given number) was studied in \cite{ChenTangSun}, where one can find detailed historical and bibliographical remarks on the subject. 

The applications of degree theory based methods require a priori estimates of solutions in appropriate function spaces. These estimates can be guaranteed if the right-hand side of a NASODDE satisfies certain  growth restrictions at infinity provided, for example, by the so-called Hartman-Nagumo condition (see, for example,  \cite{Hart,Knobloch,BebernesSchmitt,Amster-Haddad} and references therein for ODEs and  \cite{GainesMawhin} for (scalar) FDEs independent of the first derivative).

If $k$ identical oscillators described by a NASODDE are coupled $\Gamma$-symmetrically, where $\Gamma$ stands for a finite group, then symmetries of the coupling are reflected in the resulting system in the form of the $\Gamma$-equivariance of the right-hand side, i.e. it commutes with the $k$-dimensional permutation $\Gamma$-representation canonically corresponding to the coupling.  Also, 
a $\Gamma$-coupled system may exhibit additional symmetries: reversibility (cf. \cite{BW,KW}), oddness/evenness, to mention 
a few. Studying the existence of periodic solutions to such systems together with their spatio-temporal symmetric patterns, being a problem of formidable complexity, has not been given enough attention (especially, in non-variational setting). It turns out that there is a topological  tool -- the so-called {\it Brouwer equivariant degree} (see, for example, \cite{KW,AED}) --  allowing one to effectively study the problem in question in many cases. The {\it goal} of the present paper is to open a door to a systematic usage of this tool to study symmetric patterns of periodic solutions to
equivariant NASODDEs.

\medskip

\subsection{Methodology}  Brouwer degree (in compliance with the guiding function method, see \cite{Krasnosel-guiding}) as well as its different infinite dimensional generalizations (Leray-Schauder degree \cite{KZ}, coincidence degree \cite{GainesMawhin}, etc.) proved their efficiency in detecting periodic solutions to non-autonomous systems.  Theoretically, these degrees can be used in symmetric setting. 
To be more specific,  given an  orthogonal $G$-representation  $V$ (here  $G$ stands for a compact Lie group) and an admissible $G$-pair  $(f,\Omega)$  in $V$
(i.e., $\Omega\subset V$ is an open bounded $G$-invariant set and $f:V\to V$ a $G$-equivariant map without zeros on $\partial \Omega$), the Brouwer degree $d_H:=\deg(f^H,\Omega^H)$ is 
well-defined for any $H \le  G$ (here $\Omega^H:= \{x \in \Omega\, :\, hx = x\; \forall h \in H\}$ 
and $f^H:= f|_{\Omega^H}$). If for some $H$, one has $d_H\not=0$, then the existence of solutions with symmetry at least $H$ to equation $f(x)=0$ in $\Omega$ is guaranteed (similar argument can be used for infinite dimensional versions of the Brouwer degree). This approach provides a  way to determine the existence of solutions in $\Omega$, and to distinguish their different orbit types. However, this approach may be inefficient and very exhausting in the case the space $V$ contains a large number of orbit types. 




As it was mentioned above, our method is based on the usage of the Brouwer equivariant degree theory; 
 for the detailed exposition of this theory, we refer to the monographs
 \cite{AED, KW,IV-book,KB} and survey \cite{survey} (see also \cite{BKLN,BLN}). 
In short, the equivariant degree is a topological tool allowing ``counting'' orbits of solutions to symmetric equations in the same way as the usual Brouwer degree does, but according to their symmetry properties. 
To be more explicit, the equivariant degree $\gdeg(f,\Omega)$ is an element of the free $\bz$-module $A_0(G)$ generated by the conjugacy classes $(H)$ of subgroups $H$ of $G$ with a finite Weyl group $W(H)$:
\begin{equation}\label{eq:gdeg}
\gdeg(f,\Omega)=\sum_{(H)} n_H\, (H), \quad n_H\in \bz,
\end{equation} 
where the coefficients $n_H$ are given by the following recurrence formula
\begin{equation}\label{eq:rec}
n_H=\frac{d_H-\sum_{(L)>(H)} n_L \,n(H,L)\, |W(L)|}{|W(H)|},
\end{equation}
and  $n(H,L)$ denotes the number of subgroups $L'$ in $(L)$ such that $H\le L'$ (see \cite{AED}).  Also, we use the  notation 
\begin{equation}\label{eq:coeffi}
\text{coeff}^{H}(a):= n_H \;\ \text{for any} \;\;  a = \sum_{(H)} n_H(H) \in A_0(G). 
\end{equation}
One can immediately recognize a connection 
between the two collections: $\{d_H\}$ and  $\{ n_H\}$, where $H \le  G$ and $W(H)$ is finite. 
As a matter of fact,  $\gdeg(f,\Omega)$ satisfies the standard properties expected from any topological degree.  
However, there is one additional functorial property, which plays a crucial role in computations, namely the {\it multiplicativity property}. In fact, $A_0(G)$ has a natural structure of a ring  (which is called the {\it Burnside ring} of $G$), where the multiplication  $\cdot:A_0(G)\times A_0(G)\to A_0(G)$   is defined on generators by $(H)\cdot (K)=\sum_{(L)} m_L\, (L)$ with 
\begin{equation}\label{eq:mult}
m_L:=|(G/H\times G/K)_{(L)}/G|, \quad \text{ where } W(L) \text{ is finite}.
\end{equation}
The multiplicativity property for two admissible $G$-pairs $(f_1,\Omega_1)$ and   $(f_2,\Omega_2)$ means the following equality:
\begin{equation}\label{eq:mult-property}
 \gdeg(f_1\times f_2,\Omega_1\times \Omega_2)=  \gdeg(f_1,\Omega_1)\cdot  \gdeg(f_2,\Omega_2).
\end{equation}
Given a $G$-equivariant linear isomorphism $A : V \to V$, formula \eqref{eq:mult-property} combined with the equivariant spectral decomposition of $A$, reduces the computations of $\gdeg(A,B(V))$ to the computation of the so-called basic degrees $\deg_{\cV_i}$, 
which can  be ``prefabricated" in advance for any group $G$ (here $\deg_{\cV_i}:=\gdeg(-\id, B(\cV_i))$, where $\cV_i$ is an irreducible $ G$-representation and $B(X)$ stands for the unit ball in $X$).  All in all, the product property of the Brouwer equivariant degree provides a huge computational advantage in comparison with the usual Brouwer degree approach. In order to facilitate the usage of  the Brouwer  equivariant degree, a  package {\tt EquiDeg} for GAP programming allowing  effective application of the equivariant degree methods, was created by Hao-Pin Wu. This package, being available  from  \url{https://github.com/psistwu/GAP-equideg} (see  \cite{Pin}), is behind our numerical example
with dihedral symmetries considered in Subsection \ref{subsec-Didedral-example}. 

\medskip

\subsection{Overview} After the Introduction, the paper is organized as follows. In Section \ref{sec:-setting-symmetries}, we describe a class of symmetric
NASODDEs of our interest (see equation \eqref{eq:delay-eq}  and conditions \eqref{eq:tauj} and {\rm (A0)}--{\rm (A6)}), and clarify the role of several symmetry subgroups in studying symmetric patterns of periodic solutions. In Section \ref{sec:app3.2}, we establish a priori bound for solutions to problem  \eqref{eq:delay-eq} in the space $C^2(S^1;\bfV)$. In Section \ref{sec:operator-reform}, we reformulate problem  \eqref{eq:delay-eq} as a 
$G:=D_1 \times  \mathbb Z_2 \times \Gamma$-equivariant fixed point problem  in $C^2(S^1;\bfV)$ (see \eqref{eq:fun1}--\eqref{eq:fun2} and 
conditions {\rm (A4)},  {\rm (A6)}), and present an abstract equivariant degree based result (see Proposition \ref{th:app3:abstract}; $D_1$ is related to the reversing symmetry, $\mathbb Z_2$ is related to the oddness of the right-hand-side and $\Gamma$ reflects the
symmetric character of the coupling in the corresponding network ). 
This result can be effectively applied to concrete symmetric systems only if effective tools for the computation of the equivariant degree associated can be elaborated. The latter is a subject of Section \ref{sec:degree-computation}, where we combine the multiplicativity property of the equivariant degree with appropriate  equivariant spectral data of the operators involved. Based on that, we present our main symmetric result (see Theorem \ref{th:app3:main}; see also Corollary \ref{cor:app3:non-equi-main}
dealing with the non-symmetric setting). In Section \ref{sec:app3.3}, we explicitly describe a wide class of NASODDEs satisfying hypotheses of Theorem \ref{th:app3:main} and provide an illustrating example  with the group $\Gamma = D_n$ -- the dihedral group of order $2n$ (see Corollary 
\ref{cor-Dihedral}).
We conclude the paper with an Appendix related to the equivariant topology jargon and equivariant degree background.

\section{Setting and symmetries}\label{section:setting-symmetries}\label{sec:-setting-symmetries}

\subsection{Setting}\label{subsec:setting}

Take the Euclidean space  $\bfV=\bbR^n$  and  for given $m\in \bn$,  put
\[ 
\bfV^m:=\mathop{\underbrace{ \bfV\times \dots\times \bfV}}\limits_{m\times}.
\]
For  $\bold y\in \bfV^m$, denote  $\bold y:=(y^1,y^2,\dots, y^m)$, $y^j\in \bm V$, $j = 1,...,m$, and define 
\begin{equation}\label{eq:max-norm}
|\bold y|:=\max\{|y^1|,|y^2|,\dots, |y^m|\}.
\end{equation}
Assume  $\Gamma$ is a subgroup of the symmetric group $S_n$.  The group $\Gamma$ acts on vectors $x=(x_1,x_2,\dots,x_n)$ in $V$ by permuting their coordinates, i.e. 
for $\gamma\in \Gamma$, one has: 
\[\gamma(x_1,x_2,\dots,x_n):=(x_{\gamma(1)}, x_{\gamma(2)},\dots, x_{\gamma(n)}).\]
This way  $\bm V$ becomes an orthogonal $\Gamma$-representation. 
Then, the space  $\bfV^m$ equipped with the diagonal $\Gamma$-action given by $\gamma(y^1,...,y^m):= 
(\gamma  y^1,..., \gamma  y^m)$ is an orthogonal  $\Gamma$-representation.  

Take reals  $\tau_j$, $j=0, 1,2,\dots, m$, satisfying 
\begin{equation}\label{eq:tauj}
0=\tau_0<\tau_1<\dots< \tau_m<2\pi, \quad \tau_{m-j+1}=2\pi -\tau_j \;\;\text{  for $j=1,2,\dots,m$}.
\end{equation}
Notice that for  $\zeta_j:=e^{i\tau_j}$, $j=1,2,\dots,m$,  condition \eqref{eq:tauj} implies  $\overline{\zeta_j}=\zeta_{m-j+1}$.

\medskip

Take a function $f:\br \times\bfV \times\bfV^m\times \bfV\to \bfV$   
and consider the following system of 
NASODDEs:
\begin{equation}\label{eq:delay-eq} 
\ddot x(t)= f(t, x(t), \bold x_t, \dot x(t)), \;\; t\in \br,\\
\end{equation}
where for a function  $x:\br\to \bfV$, we put 
\[
\bold x_t:=(x(t-\tau_1),x(t-\tau_2), \dots, x(t-\tau_m))\in \bfV^m, \quad t\in \br.
\]
\vs

We will assume that $f$ satisfies the following conditions:
\begin{itemize}
\item[(A0)] $f$ is continuous  and 
$f(t+2\pi, x, \bold y, z) = f(t, x, \bold y, z)$ for all $t \in \mathbb R$, $x$, $z\in \bfV$, $\bold y\in \bfV^m$;
 
\item[(A1)] There exists $R>0$ such that for all $t\in \br$, for all $x$, $z\in \bfV$, $\bold y\in \bfV^m$, one has:
\[ 
|x|\ge R, \;|\bold y|\le |x|\; \text{ and } \; x\bullet z=0 \;\;  \;\;x\bullet f(t,x,\bold y,z) >0;
\]

\item[(A2)] There exists a continuous function $\phi:[0,\infty)\to (0,\infty)$ such that 
\[
 \int_0^\infty\frac{s\,ds}{\phi(s)} =\infty\quad  \text{ and }\;\;\;
\forall_{t\in \br}\; \forall_{|x|\le R}\; \forall_{|\bold y|\le R}\;\forall_{z}\;\;  |f(t,x, \bold y,z)|\le \phi(|z|);
\]
\item[(A3)] There are constants $\alpha>0$, $K>0$ such that
\[
\forall_{t\in \br}\;\forall_{|x|\le R}\;\forall_{|\bold y|\le R}\;\forall_{z\in \bfV}\;\;\;\; |f(t,x,\bold y,z)|\le \alpha (x\bullet f(t,x,\bold y, z) +|z|^2) +K; 
\]

\item[(A4)]  For all $x$, $z\in \bfV$, $\bold y=(y^1,y^2,\dots,y^m)\in \bfV^m$, one has:

\begin{itemize}
\item[(i)] $f(t,x,\bold y,-z) = f(t,x,\bold y,z)$,

\item[(ii)] $f(-t,x,\bold y,z) = f(t,x,\bold y,z)$,

\item[(iii)] $f(t,x,y^1,y^2,\dots, y^{m-1},y^m,z) = f(t,x,y^m,y^{m-1},\dots, y^2, y^1,z)$,

\item[(iv)] $f(t,-x,-\bold y,-z) = -f(t,x,\bold y,z)$;
\end{itemize}

\item[(A5)]  There exist $n\times n$-matrices $A_j$, $j=0,1,\dots, m$, such that 
\[
\forall_{t\in \br}\;\;\; \lim_{(x,\bold y,z)\to 0}\frac{f(t,x,\bold y,z)- A_0x-\sum_{j=1}^m A_jy^j}{|(x,\bold y,z)|}=0;
\]

\item[(A6)]  
$f(t,\gamma x,\gamma \bold y,\gamma z)= \gamma f(t,x,\bold y,z)$  for all $\gamma\in \Gamma$, $x$, $z\in \bfV$, $\bold y\in \bfV^m$. 

\end{itemize}

\begin{remark}\label{rem:cond-setting}
(i) (A0) is the standard condition for looking for (classical) periodic solutions to the non-autonomous system. Conditions (A1)--(A3) are classical Hartman-Nagumo conditions required to establish a priori bounds for periodic solutions to  \eqref{eq:delay-eq} (cf. \cite{Hart,GainesMawhin}).  The first three conditions in (A4) imply that system \eqref{eq:delay-eq} is time-reversible. The last condition in (A4) together with condition (A6) imply that system  \eqref{eq:delay-eq} is 
$\mathbb Z_2 \times \Gamma$-symmetric. Finally, condition (A5) guarantees the existence of linearization for  \eqref{eq:delay-eq} at zero, which is {\it independent of time $t$}. 

(ii) There are many  examples of systems  \eqref{eq:delay-eq} of practical meaning (with multiple delays) satisfying \eqref{eq:tauj}, (A0)--(A6). For instance,  one can consider time delay systems with {\it commensurate delays} which  play an important role in robust control theory (see, for example, \cite{GKhCh} and references therein). Notice that in our case, the delays $\tau_j$ are not necessarily commensurate (i.e. $\tau_j\not=\frac{2\pi j}m$).
\end{remark} 
Under the conditions \eqref{eq:delay-eq}, (A1)--(A6), assuming, in addition, that the linearization at zero of system \eqref{eq:delay-eq} is non-degenerate, our aim is to establish, using equivariant spectral properties of the matrices $A_j$, the existence of multiple non-constant $2\pi$-periodic solutions with prescribed spatio-temporal symmetries.
\vs

\subsection{Symmetries}\label{subsec:symmetries}\label{sub-sec-symmetries} 
We will study periodic solutions to \eqref{eq:delay-eq} as fixed points of the operator
associated with  \eqref{eq:delay-eq} and defined in the space  $\mathscr E:=C_{2\pi}^2(\br; \bfV)$ of $C^2$-smooth periodic functions, equipped with the norm 
\begin{equation}\label{eq:space-mathscr-E}
\|u\|:= \max\{ \|u\|_\infty,\|\dot u\|_\infty, \|\ddot u\|_\infty\}, \quad \| u\|_\infty:=\sup\{|u(t)|: t\in \br\},\quad u\in \mathscr E. 
\end{equation}
Put
\begin{equation}\label{eq:group-mathcal-G-general} 
\mathcal G:= O(2) \times \mathbb Z_2 \times \Gamma,
\end{equation}
where 
\begin{equation}\label{eq:group-mathcal-G}
O(2) = S^1 \cup S^1 \kappa, \;\; S^1= \Big\{
e^{i\theta}\simeq \begin{bmatrix}
   \cos(\theta) &-\sin(\theta)\\
   \sin(\theta)  & \cos(\theta)
 \end{bmatrix}, \; \theta \in [0,2\pi) \Big\}, \; 
 \kappa = 
 \begin{bmatrix}
   1 & 0\\
   0  & -1
 \end{bmatrix}
\end{equation}
and
\begin{equation}\label{eq:Z-2}
 \bz_2:=\{1,-1\}.
\end{equation}
Then,  $\mathscr E$ admits the natural isometric $\mathcal G$-representation given by
\begin{equation}\label{eq:mathcal-G-action}    
(e^{i\theta},\pm 1,\gamma)u(t) = \pm \gamma u(t + \theta), \quad (e^{i\theta}\kappa,\pm 1,\gamma)u(t) = \pm \gamma u(-t + \theta), \quad  t\in \br.
\end{equation}
Each of the following four subgroups of $\mathcal G$ will play an important role in what follows:

\medskip

$\bullet$ $D_1:=\{1,\kappa\} < O(2)$: this subgroup is related to the reversibility of system \eqref{eq:delay-eq};

\medskip

$\bullet$ $G:= D_1 \times \mathbb Z_2 \times \Gamma$: this subgroup classifies spatio-temporal symmetries of periodic solutions to
\eqref{eq:delay-eq} and as such gives rise to the usage of the Brouwer $G$-equivariant degree;

\medskip

$\bullet$ $D_1^z = \{(1,1,e), (\kappa,-1,e)\} < D_1 \times \bz_2 \times \{e\}$: this subgroup allows us to distinguish between constant and non-constant periodic solutions to \eqref{eq:delay-eq} (here $e$ stands for the neutral element in $\Gamma$);

\medskip

$\bullet$ $\mathfrak G := O(2) \times \{1\} \times \Gamma$: this subgroup allows us to effectively describe equivariant spectral data contributing to the computation of the $G$-equivariant degree associated to the linearization of \eqref{eq:delay-eq} at zero (by condition {\rm (A5)}, this linearization is {\it autonomous}, and as such admits a $\mathfrak G$-equivariant treatment);  at the same time, since \eqref{eq:delay-eq} is a {\it non-autonomous} system, the $S^1$-component  of $O(2)$ is irrelevant to classifying spatio-temporal patterns of periodic solutions to   \eqref{eq:delay-eq} (leaving a scene to $D_1$).

\vs 

\section{A Priori Bound}\label{sec:app3.2}

\subsection{Auxiliary lemmas} The  {\it a priori} bound for periodic solutions to a second order system of ODEs was derived in the classical book by P. Hartman \cite{Hart} (see also \cite{Knobloch,Nagumo,BebernesSchmitt,Amster-Haddad}). The argument required for our setting (including multiple delays) is close to the one used to treat ODEs. However, since to the best of our knowledge, there is no paper
combining (multiple) delays with the Hartman-Nagumo condition, we decided, for the sake of completeness, to include a proof covering this case. {\bf Hereafter},  $p:=2\pi$.

\vs 
\begin{theorem}\label{th:apriori} Let $f:\br \times\bfV \times\bfV^m\times \bfV\to \bfV$ satisfy {\rm (A0)---(A3)}.
Then, there exists a constant $C > 0$ such that  any $p$-periodic solution $x(t)$ to system \eqref{eq:delay-eq}  satisfies 
\begin{equation}\label{eq:ap-1}
\forall_{t\in \br}\;\;\;|x(t)|,\,|\dot x(t)|,\,|\ddot x(t)|<C.
\end{equation} 
\end{theorem}
\vs 
To prove Theorem \ref{th:apriori}, we need the following lemmas.
\vs

\begin{lemma}\label{lem:1} Under the assumptions {\rm (A0)---(A3)}, any $p$-periodic solution $x(t)$ to \eqref{eq:delay-eq}  satisfies 

\[\forall_{t\in \br}\;\;\; |x(t)| \leq R,\]
where the constant $R$ is given in condition {\rm (A1)}.
\end{lemma}
	\begin{proof} Assume that  $x(t)$ is a $p$-periodic solution to \eqref{eq:delay-eq} and consider the function $r(t):=|x(t)|^2$, $t\in \br$. Obviously,
\begin{equation}\label{eq:2-3}
\dot r(t) =2x(t)\bullet \dot x(t), \qquad \ddot r(t) =2 x(t)\bullet \ddot x(t)+2|\dot x(t)|^2.
\end{equation} 
Assume that $r(t)$ achieves its maximum at $t_{o}$. Then,  $\dot r(t_o) =0$ and $\ddot r(t_o)\le 0$, which by condition (A1), implies that $|x(t_{o})|\le R$. Therefore, $ |x(t)| \leq R$ for all $t\in \br$.
\end{proof}

\vs
\begin{lemma}\label{lem:2} Under the assumptions { \rm (A0)---(A3)}, there exists a constant $M:=M(\alpha, R ,K)$ such that for any solution $x(t)$
	to \eqref{eq:delay-eq}, one has
	\begin{equation}\label{eq:ap-2}
	\forall_{t\in \br}\;\;\;|\dot x(t)|\le M.
\end{equation}
	\end{lemma}
\begin{proof}
	Let $x(t)$ be a $p$-periodic solution to \eqref{eq:delay-eq}. By Lemma \ref{lem:1}, one has $|x(t)|\le R$  for all $t\in \br$. Hence
(see \eqref{eq:max-norm}),  $|\bold x_t| \leq R$ for all $t\in \br$. Combining this with (A2) yields: 
	\[\forall_{t\in \br}\;\; 
	|\ddot x(t) |= |f(t, x(t),\bold x_t,\dot x(t))| \le \phi(|\dot x(t)|).
	\]
In addition, by condition (A3), there exist constants $\alpha >0$, $K>0$, such that
\begin{align*}
	|\ddot x(t)| &= |f(t, x(t),\bold x_t,\dot x(t))|\le \alpha \Big(x(t)\bullet f(t,x(t), \bold x_t,\dot x(t))+|\dot x(t)|^2\Big) +K=\frac \alpha 2 \ddot r(t)+K
\end{align*}
(cf. \eqref{eq:2-3}), hence
\begin{equation}\label{eq:ap-3}
\forall_{t\in \br}\;\; \;\; |\ddot x(t)|\le \frac{\alpha}{2}\ddot r(t)+K.
\end{equation}
	
			Next, by using integration by parts and $p$-periodicity of $x(t)$, one obtains 
\begin{align*}
 \int_t^{t+p} (t+p-s)\ddot x(s)ds&=  (t+p-s)\dot x(s)\Big|_{t}^{t+p} +\int_{t}^{t+p} \dot x(s)ds= x(t+p)-x(t) - p \dot x(t)= -p\dot x(t), 
 \end{align*}
i.e.
\begin{equation}\label{eq:ap-4}
\forall_{t\in \br}\;\;\;   p\dot x(t)= - \int_t^{t+p}\;\; (t+p-s)\ddot x(s)ds.
\end{equation}
In particular,
\begin{equation}\label{eq:new1}
p\dot x(0)\; =\;-\int_0^{p} \; (p-s)\ddot x(s)ds.
\end{equation}
Combining \eqref{eq:new1} with \eqref{eq:ap-3} yields:
\begin{align*}
p|\dot x(0)|\;
&\le \;\int_0^{p}\;(p-s)|\ddot x(s)|ds\le \; \int_0^{p}\;(p-s)\Big(\frac \alpha 2\ddot r(s)+K\Big)ds\\
&=\int_0^{p}\;(p-s)\frac{\alpha}{2}\ddot r(s)ds+K\int_0^{p}\;(p-s)ds=-\frac{\alpha}{2}p\dot r(0)+\frac {1}{2}Kp^2,
\end{align*}
i.e.
\begin{equation}\label{eq:ap-6}
 p|\dot x(0)|\; \le  -\frac{\alpha}{2}p\dot r(0)+\frac {1}{2}Kp^2.
\end{equation}
Similarly, 
\begin{equation}
p\dot x(t)=-\int_{t-p}^t (t-p-s)\ddot x(s)ds.
\end{equation}
In particular,
\begin{equation}\label{eq:new2}
 p\dot x(0)=\int_{-p}^0 (p+s)\ddot x(s)ds.
\end{equation}
Combining \eqref{eq:new2} with \eqref{eq:ap-3} yields:
\begin{equation}\label{eq:ap-7}
p|\dot x(0)|\le \frac {\alpha}{2}p\dot r(0)+\frac {1}{2}Kp^2.
\end{equation}
Adding \eqref{eq:ap-6} and \eqref{eq:ap-7}  leads to 
\begin{equation}\label{eq:ap-0}
2p|\dot x(0)|\le Kp^2\quad \Leftrightarrow\quad 
|\dot x(0)|\le \frac {1}{2}Kp.\end{equation}
Moreover, by \eqref{eq:ap-4}, one has: 
\begin{align*}
p|\dot x(t)|
&\le \int_t^{t+p}\;(t+p-s)|\ddot x(s)|ds\le \int_t^{t+p}\;(t+p-s)\Big(\frac {\alpha}{2}\ddot r(s)+K\Big)ds=-\frac {\alpha}{2}p\dot r(t)+\frac {1}{2}Kp^2.
\end{align*}
The last inequality  together with condition (A2) imply
\begin{equation} \label{eq:ap-8}
\frac {\dot x(t)\bullet \ddot x(t)}{\phi(|\dot x(t)|)}\le \frac {|\dot x(t)\bullet \ddot x(t)|}{\phi(|\dot x(t)|)}\le \frac{|\dot x(t)||\ddot x(t)|}{\phi(|\dot x(t)|)}\le |\dot x(t)|\le \frac {1}{2}Kp-\frac {\alpha}{2}\dot r(t).
\end{equation}
Next, integrating inequality \eqref{eq:ap-8} on  $[0,t]$, where $t \in [0,p]$, one obtains:
\begin{align}
\left|\int_0^{t}\;\frac {\dot x(s)\bullet \ddot x(s)}{\phi(|\dot x(s)|)}ds\right|& \le \int_0^{t}\left[\frac {1}{2}Kp-\frac {\alpha}{2}\dot r(s)\right]ds
\nonumber\\
&=\frac {1}{2}Kpt-\frac {\alpha}{2}[r(t)-r(0)]\le \frac {K}{2}p^2+\frac {\alpha}{2}2R^2=\frac {K}{2}p^2+\alpha R^2.\numberthis 
\label{eq:integr-ineq}
\end{align}
On the other hand, using the substitution $u:= |\dot{x}(s)|$, one obtains:
\begin{equation}\label{eq:integr-rep}
\int_0^{t}\;\;\frac {\dot x(s)\bullet \ddot x(s)}{\phi (|\dot x(s)|)}ds \;=\int_{ |\dot x(0)|}^{|\dot x(t)|}\frac {u du}{\phi (u)}.
\end{equation}
Put $\displaystyle \Phi (\omega):=\;\int_0^{\omega}\;\frac {u du}{\phi(u)}$,  so we have
\begin{equation}\label{eq:estim-Phi}
\left|\int_{ |\dot x(0)|}^{|\dot x(t)|}\;\;\frac {sds}{\phi (s)}\right|= |\Phi(|\dot x(t)|) - \Phi(|\dot x(0)|)|.
\end{equation}
Combining \eqref{eq:integr-ineq}-\eqref{eq:estim-Phi} yields:
\begin{align*}
 |\Phi(|\dot x(t)|)-\Phi (|\dot x(0)|)|\le \frac{K}{2}p^2+\alpha R^2,
\end{align*}
and consequently (see \eqref{eq:ap-0}), 
\begin{align}\label{eq:important}
\Phi (|\dot x(t)|)&\le \frac {1}{2}Kp^2+\alpha R^2+\Phi (|\dot x(0)|)|      \le \frac{1}{2}Kp^2+\alpha R^2+\Phi\left(\frac{1}{2}Kp\right).
	 \end{align}
By (A2),
$\displaystyle \lim_{w \to \infty} \Phi(w) = \infty$, hence $\Phi:[0,\infty)\to [0,\infty)$ is a continuous monotonic bijective function. Combining
this with \eqref{eq:important} yields:
\[	|\dot x(t)|\le \Phi^{-1}\left[\frac{1}{2}Kp^2+\alpha R^2+\Phi\left(\frac{1}{2}Kp\right)\right]=:M.
\]
Therefore,  there exists a constant $M>0$ such that 
\[
\forall_{t\in \br}\;\;\;|\dot x(t)|\le M.
\]
\end{proof}

\subsection{Proof of Theorem \ref{th:apriori}}
Put  
\[N:=\; \max \{\phi(s): s\in [0,M]\}.\]
Then, by Lemmas  \ref{lem:1}, \ref{lem:2} and condition (A2), one has: 
\[|\ddot x(t)|= |f(t,x(t),\bold x_t,\dot x(t))|\le \phi(|\dot x(t)|)\le N\; \; \text { for all }\; t\in \br.\]
Then, clearly, the constant  $C := R+M+N+1$ satisfies \eqref{eq:ap-1}.
\endproof

\vs
\vs
\section{Setting System \eqref{eq:delay-eq}  in Functional Spaces}\label{sec:operator-reform} 
\subsection{Operator reformulation and deformation}
Together with the space $\mathscr E=C_{2\pi}^2(\br; \bfV)$ (see \eqref{eq:space-mathscr-E}), we will use the following functional spaces of $2\pi$-periodic functions:   
\begin{itemize}
\item[(i)] the space $\wt {\mathscr C}:= C_{2\pi}(\br; \bfV\times \bfV^m\times \bfV)$ of continuous $ \bfV\times \bfV^m\times \bfV$-valued functions with the usual $\sup$-norm;
\item[(ii)]  the space ${\mathscr C}:= C_{2\pi}(\br; \bfV)$ of continuous $ \bfV$-valued functions with the usual $\sup$-norm.
\end{itemize}
Obviously, formula \eqref{eq:mathcal-G-action} (see also \eqref{eq:group-mathcal-G-general}--\eqref{eq:Z-2}) defines on $\wt {\mathscr C}$ and ${\mathscr C}$ isometric Banach $\mathcal G$-representations. To be more specific regarding the $\kappa$-action on $\wt {\mathscr C}$, observe:
\begin{equation}\label{eq:kappa-action-u_t}
\kappa \bold{u}_t = (u(-(t - \tau_1)), u(-(t - \tau_2)),...,u(-(t - \tau_m))) = (u(-t + \tau_1), u(-t + \tau_2),...,u(-t + \tau_m)).
\end{equation}

Next, define the following operators: 
\begin{alignat*}{3}
i &:\mathscr E\to \mathscr C, \quad& i u(t)&:= u(t),\\
L&:\mathscr E\to \mathscr C, \quad& Lu&:= \ddot u(t) - i u(t),\\
J&:\mathscr E\to \wt{\mathscr C}, \quad &(Ju)(t)&:=(u(t), \bold u_t,\dot u(t)),\\
N&: \wt{\mathscr C}\to \mathscr C, \quad &N(x,\bold y,z)(t)&:=f(t,x(t),\bold y(t), z(t))-x(t).
\end{alignat*}
Notice that $L$ is an isomorphism, $i$ and $J$ are compact and $N$ is continuous. 
Clearly, system \eqref{eq:delay-eq}  is equivalent to the equation
\begin{equation}\label{eq:fun1}
Lx=N(Jx), \quad x\in \mathscr E.
\end{equation}
Put $\mathscr F(x):=x-L^{-1}N(Jx)$, $x\in \mathscr E$. Obviously,  $\mathscr F$ is a completely continuous field on $\mathscr E$ and  \eqref{eq:fun1} admits an operator reformulation:
\begin{equation}\label{eq:fun2}
\mathscr F(x)=0, \quad x\in \mathscr E.
\end{equation}
\vs
Consider the deformation $\mathscr F_\lambda: \mathscr E\to \mathscr E$ given by $\mathscr F_\lambda(x):= x-\lambda L^{-1}N(Jx)$, $x\in \mathscr E$, $\lambda\in [0,1]$. 
Then, for each $\lambda \in [0,1]$, the equation 
\[
\mathscr F_\lambda(x)=0
\]
is equivalent to the system 
\begin{equation}\label{eq:delay-eqnew} 
\begin{cases}
\ddot x(t)= \lambda f(t, x(t), \bold x_t, \dot x(t))+(1-\lambda) x(t), \;\; t\in \br\\
x(t)=x(t+2\pi),\;\; \dot x(t+2\pi)=\dot x(t).
\end{cases}
\end{equation}
For a fixed $\lambda\in (0,1]$, define the function 
\begin{equation} \label{eq:fun3}
\mathfrak f(t,x,\bold y,z):= \lambda f(t, x, \bold y, z)+(1-\lambda) x, \quad x,\,z\in \bfV, \; \bold y\in \bfV^m, \; t\in \br.
\end{equation}
\vs

\begin{lemma} \label{lem:f1}
If the function $f$ satisfies conditions {\rm (A0)--(A3)} with the constants $R$, $\alpha$, $K$  and the function $\phi$, then the function $\mathfrak f$ 
given by \eqref{eq:fun3} satisfies conditions {\rm (A0)--(A3)} with the constants $R$, $\alpha$, $K':=K+R$  and the function $\hat{\phi}(s)=\phi(s)+R$.
\end{lemma}
\begin{proof} Obviously, $\mathfrak f$ satisfies {\rm (A0)}.  
 Take $x$, $\bold y$ with $|x|\ge R$, $|\bold y|\le |x|$ and $x\bullet z=0$. Then, 
\begin{align*}
x\bullet \mathfrak f(t, x, \bold y, z)=x\bullet (\lambda f(t, x, \bold y, z)+(1-\lambda) x)=\lambda x\bullet f(t, x, \bold y, z)+(1-\lambda)x\bullet x\ge 0,
\end{align*}
thus, $\mathfrak f$ satisfies {\rm (A1)}.
Also, 
for any $x$, $\bold y$  with $|x|,|\bold y | \le R$, and any $z$, one has:
\begin{align*}
|\mathfrak f(t,x,\bold y, z)|&= |\lambda f(t,x, \bold y,z)+(1-\lambda)x|\le \lambda |f(t,x,\bold y,z)|+(1-\lambda)|x|\\
&\le \lambda\phi(|z|)+(1-\lambda)|x|\le \phi(|z|)+|x|\le \phi(|z|)+R =: \hat{\phi}(|z|).
\end{align*}
Thus, 
$\mathfrak f(t,x,\bold y, z)$ satisfies {\rm (A2)}.
Finally, 
\begin{align*}
|\mathfrak f(t,x,\bold y, z)|
&= |\lambda f(t,x, \bold y,z)+(1-\lambda)x|\le \lambda |f(t,x,\bold y,z)|+(1-\lambda)|x|\\
&\le \lambda [\alpha (x\bullet f(t,x,\bold y, z)+ |z|^2)+K]+(1-\lambda)|x|\\
&= \lambda \alpha x\bullet f(t,x,\bold y,z)+\lambda \alpha |z|^2+\lambda K+(1-\lambda)|x|\\
&\le \alpha \lambda x\bullet f(t,x,\bold y,z) +\alpha |z|^2+K+|x|\\
&\le \alpha \lambda x\bullet f(t,x,\bold y,z)+ \alpha (1-\lambda)|x|^2+\alpha |z|^2+K+|x|\\
&\le \alpha \lambda x\bullet f(t,x,\bold y,z) +\alpha (1-\lambda)|x|^2+\alpha |z|^2+K+R\\
&=\alpha [\lambda x\bullet f(t,x,\bold y,z)+ (1-\lambda)x\bullet x+|z|^2]+K+R\\
&=\alpha [ x\bullet (\lambda f(t,x,\bold y,z)+ (1-\lambda)x)+|z|^2]+K+R\\
&=\alpha [ x\bullet \mathfrak f(t,x,\bold y,z)+|z|^2]+K+R.\\
\end{align*}
Thus, {\rm (A3)} is satisfied with $K'= K+R$.
 \end{proof}
\vs

Lemma \ref{lem:f1} together with Theorem \ref{th:apriori} immediately imply
\begin{lemma} \label{lem:f2}    If  the function $f:\br \times\bfV \times\bfV^m\times \bfV\to \bfV$  satisfies the assumptions {\rm (A0)---(A3)}, then there exists 
 a constant $C>0$ such that  for any $\lambda\in [0,1]$ and $x\in \mathscr E$, one has: 
 \[
 \mathscr F_\lambda (x)=0 \;\;\; \Rightarrow \;\;\; \|x\|<C.
 \]
\end{lemma}

%
\subsection{Properties  of the map $\mathscr F$}	
Define the operator $\mathfrak A: \wt{\mathscr C}\to \mathscr C$ by
\begin{equation}\label{eq:oper:A}
\mathfrak A(x,\bold y,z)(t)= A_0x(t)+\sum_{j=1}^m A_j y^j(t), \quad (x,\bold y,z)\in \wt{\mathscr C},  
\end{equation}
and let the operator $\mathscr A:\mathscr E\to \mathscr E$ be given by 
\begin{equation}\label{eq:linA}  
\mathscr A(x):=x-L^{-1}\Big( \mathfrak A (Jx) - i x \Big), \quad x\in \mathscr E.
\end{equation}


\vs
In the statement following below, we summarize properties of the map $\mathscr F$.
\vs
\begin{proposition}\label{pro:app3:mapF}
Assume that the function   $f:\br \times\bfV \times\bfV^m\times \bfV\to \bfV$  satisfies the conditions 
{\rm (A0)--(A6)}. Assume, in addition, that \eqref{eq:tauj} takes place. Then:
\begin{itemize}
\item[(a)] The map $\mathscr F$ is a $G$-equivariant completely continuous field;
\item[(b)] There exists a sufficiently large $C>0$ such that $\mathscr F$ is $B_C(0)$-admissibly $G$-homotopic to $\id$;
\item[(c)] The map $\mathscr F$ is differentiable at $0$,  $D\mathscr F(0)=\mathscr A$ and is $G$-equivariant;
\item[(d)] If $\mathscr A$ is an isomorphism, then there exists $\ve>0$ such that $\mathscr F$ is $B_\ve(0)$-admissibly $G$-homotopic to $\mathscr A$.
\end{itemize}
\end{proposition}

\begin{proof} 
\noi(a) Condition {\rm (A0)} together with the compactness of $i$ and $J$ imply the complete continuity of the field $\mathscr F$. By {\rm (A4)} (resp. {\rm (A6)}), $\mathscr F$ is $\bz_{2}$-equivariant (resp. $\Gamma$-equivariant). Let us check that $\mathscr F$ is $D_{1}$-equivariant. In fact, for all $t\in \br$ and $u\in \mathscr E$, one has (we skip $i$ and $J$ for the sake of simplicity):
	{\footnotesize 
		\begin{align*}
		\mathscr F(\kappa u)(t)
		&= \kappa u(t)-L^{-1}\Big(f(t,\kappa u(t), \kappa \bold u_t, \kappa \dot u(t))-\kappa u(t)\Big)\\
		&=u(-t) -L^{-1}\Big( f(t, u(-t),u(-t + \tau_1), \dots, u(-t + \tau_m)), -\dot u(-t))-u(-t)\Big) \; (\text{by \eqref{eq:mathcal-G-action} and \eqref{eq:kappa-action-u_t}})\\
		&=u(-t) -L^{-1}\Big( f(t, u(-t),u(-(t  + 2\pi - \tau_1)),\dots, u(-(t + 2\pi - \tau_m)),- \dot u(-t))-u(-t)\Big) \; \text{(by periodicity of $u$)}\\
		&=u(-t) -L^{-1}\Big( f(t, u(-t),u(-(t  + 2\pi - \tau_1)),\dots, u(-(t + 2\pi - \tau_m)),  \dot u(-t))-u(-t)\Big) \; \text{(by (A4)(i))} \\
		&=u(-t) -L^{-1}\Big( f(t, u(-t),u(-t-\tau_m),\dots, u(-t-\tau_1),\dot u(-t))-u(-t)\Big) \; \text{(by \eqref{eq:tauj})}\\
		&=u(-t) -L^{-1}\Big( f(t, u(-t),u(-t-\tau_1),\dots, u(-t-\tau_m),\dot u(-t))-u(-t)\Big)  \; \text{(by (A4)(iv))} \\
		&=u(-t) -L^{-1}\Big( f(-t, u(-t),u(-t-\tau_1),\dots, u(-t-\tau_m),\dot u(-t))-u(-t)\Big)  \; \text{(by (A4)(ii))} \\
		&=\kappa u(t)-\kappa L^{-1}\Big( f(t,u(t),u(t-\tau_1), \dots, u(t-\tau_m),\dot u(t))-u(t)\Big)\; \text{(by \eqref{eq:mathcal-G-action})}\\
		&=\kappa\Big(u(t)-L^{-1}\big(f(t,u(t),\bold u_t,\dot u(t)) -u(t)\Big)\\
		&=\kappa \mathscr F(u)(t).
		\end{align*}
		}

\noi(b) Follows from Lemma 
\ref {lem:f2}.\\
(c) Follows from {\rm (A5)}, boundedness of $L^{-1}$ and $G$-equivariance of $\mathscr F$.\\
(d)  Define the linear homotopy
$H:[0,1]\times \mathscr E\to \mathscr E$ such that 
\[H(\lambda, u):=(1-\lambda)\mathscr Au+\lambda\mathscr F(u).\]
Suppose for contradiction, that there exists a sequence $\{\lambda_n, \,u_n\}$, such that $\lambda_n\to \lambda_0$, $u_n\to  0$, $u_n \neq 0$, 
and 
\begin{equation}\label{eq:vspo1}
0=H(\lambda_n,u_n) = (1-\lambda_n)\mathscr Au_n+\lambda_n\mathscr F(u_n) =\mathscr Au_n+\lambda_n (\mathscr F(u_n)-\mathscr Au_n).
\end{equation}
Dividing \eqref{eq:vspo1} by $\|u_n\| \neq 0$ yields:
\[
0=\mathscr A\frac {u_n}{\|u_n\|}+\lambda_n\frac {\mathscr F(u_n)-\mathscr Au_n}{\|u_n\|}.
\]
Put $v_n:=\frac {u_n}{\|u_n\|}$. Then:
\begin{equation}\label{eq:vspo2}
0=\mathscr Av_n+\lambda_n \frac {\mathscr F(u_n)-\mathscr Au_n}{\|u_n\|}.
\end{equation}
Since $ \|u_n\|\to 0$ and $\{\lambda_n\}$ is bounded,  item (c) implies:
\[
\lim_{n\to \infty}\;\;\; \frac {\mathscr F(u_n)-\mathscr Au_n}{\|u_n\|}=0.
\]
Therefore (see \eqref{eq:vspo2}), one has: 
\begin{equation}\label{eq:vspo4}
\lim_{n\to \infty}\mathscr Av_n=0.
\end{equation}
On the other hand, 
\begin{equation}\label{eqvspo5}
\mathscr A = \id - R \quad  \left(R:=L^{-1} \circ (\mathfrak A \circ J - i \circ \id)\right)
\end{equation} 
is a compact linear field ($i$ and $J$ are compact). Hence, by passing to a subsequence, one can assume without loss of generality 
that $Rv_n\to v_0$. But this implies $v_n\to v_0$ (see \eqref{eq:vspo4}) and \eqref{eqvspo5}), that is $v_0\in \ker\mathscr A$. However, $\|v_0\|=\lim_{n\to \infty}\|\frac{u_n}{\|u_n\|}\|=1$, and one arrives at the contradiction to the fact that $
\mathscr A$ is an isomorphism.
\end{proof}

\subsection{Abstract equivariant degree based result}  

Under the assumptions  \eqref{eq:tauj} and  {\rm (A0)--(A6)}, the $G$-equivariant degree $\gdeg(\mathscr A,B(\mathscr E))  \in A_0(G)$ is correctly defined
provided that $\mathscr A$ is an isomorphism (here $B(\mathscr E)$ denotes the unit ball in $\scrE$ and $A_0(G)$ stands for the Burnside ring of $G$).
Put 
\begin{equation}\label{eq5}
\omega := (G) - \gdeg(\mathscr A,B(\mathscr E)).
\end{equation}
We are now in a position to formulate the abstract result.
\begin{proposition} \label{th:app3:abstract} Assume that \eqref{eq:tauj} takes place and the function   $f:\br \times\bfV \times\bfV^m\times \bfV\to \bfV$  satisfies the conditions {\rm (A0)--(A6)}.  Assume, in addition, that $\mathscr A:\mathscr E\to \mathscr E$ is an isomorphism (see \eqref{eq:linA}). Assume, finally, that 
\begin{equation}\label{eq:vspo7}
\omega = n_1(H_1) +n_2(H_2) +\dots n_k(H_k), \quad n_j\not=0, \; (H_j) \in\Phi_0(G), \; j =1,2,\dots,k
\end{equation}
(cf. \eqref{eq5}). Then:
\begin{itemize}
\item[(a)] for every $j=1,2,\dots,k$ there exists a solution $u\in \mathscr E$ to \eqref{eq:delay-eq}  satisfying $G_u\ge H_j$;
\item[(b)] if, in addition, $H_j\ge D_1^z:= \{(1,1,e),(\kappa, -1, e)\}$, then $u$ is a non-constant periodic solution (here $e\in \Gamma$ stands for the neutral element of $\Gamma$).
\end{itemize}
\end{proposition}
\begin{proof}
(a) Take the constants $C > 0$ and $\ve>0$ provided by  Proposition \ref{pro:app3:mapF}. Combining the 
equivariant homotopy invariance of the Brouwer $G$-equivariant degree with  Proposition \ref{pro:app3:mapF}((a),(b)), one obtains
 $\gdeg(\mathscr F,B_C(0)) =\gdeg(\id,B_C(0))=(G)$. By the same reason, Proposition \ref{pro:app3:mapF}((a),(c),(d)) implies
 $\gdeg(\mathscr F,B_\ve(0)) = \gdeg(\mathscr A,B_\ve(0)) = \gdeg(\mathscr A,B(\mathscr E))$. Put $\Omega:=B_C(0)\setminus \overline{B_\ve(0)}$. 
 Then, by additivity of the Brouwer $G$-equivariant degree,  one has: 
\begin{equation}\label{eq:vspo8}
\gdeg(\mathscr F,\Omega)=\gdeg(\mathscr F,B_C(0))-\gdeg(\mathscr F,B_\ve(0))=(G)-\gdeg(\mathscr A,B_\ve(0)).
\end{equation}
Combining \eqref{eq:vspo8} with \eqref{eq:vspo7} and the existence property of the Brouwer $G$-equivariant degree yields statement (a). 

\medskip

(b)  Condition $H_j\ge D_1^z:= \{(1,1,e),(\kappa, -1, e)\}$ implies:
\[  
\forall_{t\in \br}\;\;\;(\kappa,-1)u(t)=-u(-t)=u(t). 
\]
Hence, $u$ is an odd function. Since $0 \not\in \Omega \ni u$, it follows that $u$ is not a constant function.  
\end{proof}

\section{Computation of $\omega$ and main results}\label{sec:degree-computation}
Proposition \ref{th:app3:abstract} reduces the study of problem \eqref{eq:delay-eq}  to computing the equivariant invariant $\omega$ (cf. \eqref{eq5} and 
\eqref{eq:vspo7}). In this section, we, first, analyze the equivariant spectral data required for the computation of 
$ \gdeg(\mathscr A,B(\mathscr E))$ (see \eqref{eq5}). Next, by linking these data to basic $G$-equivariant degrees, we present effective results on the existence and symmetric properties of periodic solutions to problem \eqref{eq:delay-eq}.

\subsection{$G$-isotypic decomposition of $\mathscr E$}
We will use the notations introduced in Subsection \ref{sub-sec-symmetries}. 
With an eye toward  determining equivariant spectral data of $\mathscr A$, consider, first, the $O(2)$-isotypic decomposition of the space $\mathscr E$
corresponding to its Fourier modes:  
\begin{align}\label{dcp}
\scrE=\overline{\bigoplus\limits_{k=0}^{\infty}{\mathscr E}_k},\quad{\mathscr E}_k:=\{\cos(kt)u+ \sin(kt) v : u,\, v \in \bfV\},
\end{align}
where, for $k\in \bn$, the representation ${\mathscr E}_k$ is equivalent to the complexification $\bfV^c := \bfV \oplus i \bfV$ of 
$\bf V$ (considered as a {\it real} $O(2)$-representation), where the rotations $e^{i\theta}\in SO(2)$ act on vectors $\bold z\in \bfV^c$ by $e^{i\theta}(\bold z) :=e^{-ik\theta}\cdot \bold z$ (here `$\cdot$'  stands for complex multiplication) and $\kappa \bold z:=\overline {\bold z}$ 
(cf. \eqref{eq:group-mathcal-G}--{\eqref{eq:mathcal-G-action}). Indeed, the linear isomorphism $\vp_k : \bfV^c\to {\mathscr E}_k$ given by 
\begin{equation}\label{eq:complexification}
\vp_k(x+iy):= \cos(kt) u + \sin(kt) v, \quad u,\, v\in \bfV,
\end{equation}
is $O(2)$-equivariant. Clearly, ${\mathscr E}_0$ can be identified with $\bfV$ with the trivial $O(2)$-action (denote by $\cW_0$ the trivial 
$one$-dimensional $O(2)$-representation). Also, ${\mathscr E}_k$, $k = 1,2,\ldots$, 
is modeled on the irreducible $O(2)$-representation $\cW_k\simeq\mathbb{R}^2$, where $SO(2)$ acts by $k$-folded rotations and $\kappa$ acts by complex conjugation. It follows from \eqref{dcp}--\eqref{eq:complexification} and definition of the operator $L$ that 
\begin{equation}\label{eq:action-L}
 L|_{{\mathscr E}_0}=-\id  \quad \quad \rm{and} \quad L|_{{\mathscr E}_k} = - (k^2+1) \id : \bfV^c \to \bfV^c \quad (k > 0). 
 \end{equation}
Next, we will refine the $O(2)$-isotypic decomposition \eqref{dcp} to the $\mathfrak G := O(2) \times \Gamma$-isotypic decomposition. To this end,
consider 
the  $\Gamma$-isotypic decomposition of $\bfV$, i.e.
\begin{equation}\label{eq:V-decompo-Gamma}
\bfV=\bfV_{0}\oplus \bfV_{1}\oplus\dots \oplus \bfV_{\mathfrak r},
\end{equation}
where $\bfV_l$ is modeled on $\cU_l$ ($l=0,1,\dots,\mathfrak r$;\, $\dim \bfV_l > 0$). Then, the irreducible $\mathfrak G$-representations $\cW_k\otimes \cU_l$ $(k = 0,1,..; \; 
l = 0,...,\tau)$ suggest  the $\mathfrak G$-isotypic components in $\scrE$. More precisely, for the (trivial) $O(2)$-component  $\mathscr E_0\equiv \bfV $, one has the  $\mathfrak G$-isotypic decomposition  
    \begin{equation}\label{iso-0}
     \mathscr E_0=\mathscr E_{0,0}\oplus \mathscr E_{1,0}\oplus\dots \oplus \mathscr E_{\mathfrak r,0},
    \end{equation}
where the isotypic component $\mathscr E_{l,0}$  is equivalent to $\bfV_l$, $l=0,1,\dots,\mathfrak r$ (with the trivial $O(2)$-action), i.e. 
   $\mathscr E_{l,0}$  is modeled on the irreducible representation $\cU_l$.  At the same time, for $k>0$,
 the $\mathfrak G$-isotypic decomposition of $\mathscr E_k$ is 
    \begin{equation}\label{eq:iso-k}
      \mathscr E_k=\mathscr E_{0,k}\oplus \mathscr E_{1,k}\oplus\dots \oplus \mathscr E_{\mathfrak r,k},
    \end{equation}
    where 
    \begin{equation}\label{eq-isotyp-O2-Gamma}
  \mathscr E_{l,k}:=   \{\cos(kt)u_o+ \sin(kt) v_o:u_o,\, v_o\in \bfV_l\}.
    \end{equation}
In order to obtain from \eqref{dcp}--\eqref{eq-isotyp-O2-Gamma} the $G:=D_1 \times \mathbb Z_2 \times \Gamma$-decomposition of 
$ \mathscr E$, put $\mathscr G :=  D_1\times \bz_2 \simeq D_1\times \bz_2\times \{e\}\le G$ and denote   by $\cV^+=\br$ and $\cV^-=\br$ two irreducible 
$\mathscr G$-representations given by:  
\[
\forall_{x\in \cV^\pm} \;\;\; (\kappa,1)x=\pm x, \quad  (1,-1)x = - x.
\]
Then, clearly, $\mathscr G_x=D_1$ for any non-zero $x\in \cV^+$, and $\mathscr G_x=D_1^z$ for any non-zero $x\in \cV^-$. Put 
\begin{equation}\label{eq:cU-l-pm}
\cU_l^\pm:= \cV^\pm \otimes \cU_l \quad\quad (l=0,1,\dots,\mathfrak r).
\end{equation}  
Then (see \eqref{iso-0}), the subspace $\mathscr E_{l,0}$ is isotypic and modeled on $\cU_l^+$.
At the same time, for $k>0$ (see \eqref{eq:iso-k} and \eqref{eq-isotyp-O2-Gamma}), one has: 
\[
\mathscr E_{l,k}=\mathscr E_{l,k}^+\oplus \mathscr E_{l,k}^-,
\]
where 
\begin{equation}\label{eq:vspo10}
\mathscr E^{+}_{ l,k}=\{ \cos(kt) u_o: u_o\in \bfV_l\}, \quad  \mathscr E^{-}_{l,k}=\{\sin(kt)v_o: v_o\in \bfV_l\}
\end{equation}
and are modelled on $\cU_l^+$ and $\cU_l^-$ respectively ($l=0,1,\dots,\mathfrak r$).
Thus, the $G$-isotypic decomposition of $\mathscr E$ is given by
\begin{equation}\label{eq:isoE}
\mathscr E=\bigoplus_{l=0}^{\mathfrak r} \Big(\mathscr E^{+,l}\oplus \mathscr E^{-,l}\Big),
\end{equation}
where 
\begin{equation}
\mathscr E^{+,l} = \mathscr E_{l,0}\oplus \overline{ \bigoplus_{k=1}^\infty  \mathscr E_{l,k}^+ },\quad\quad  
\mathscr E^{-,l} = \overline{ \bigoplus_{k=1}^\infty  \mathscr E_{l,k}^- },
\end{equation}
$\mathscr E_{l,0}$ is modeled on   $\cU_l^+$
and $\mathscr E_{l,k}^+ $ and 
$\mathscr E_{l,k}^- $ are described in \eqref{eq:vspo10} and are modelled 
on $\cU_l^+$ and $\cU_l^-$ respectively ($l=0,1,\dots,\mathfrak r$).

\subsection{Spectrum of $\mathscr A$}
By Proposition \ref{pro:app3:mapF}(c), the operator $\mathscr A$ is $G$-equivariant. Since the linearization of \eqref{eq:delay-eq} at the origin is
autonomous (cf. {\rm (A5)} and \eqref{eq:oper:A}--\eqref{eq:linA}),  $\mathscr A$ is $\mathfrak G$-equivariant, which can be used 
to determine spectral properties of $\mathscr A$. To be more specific, the $\Gamma$-equivariance of the matrices $A_j$  (see assumptions (A5) and (A6)) implies $A_j(\bm V_l)\subset \bm V_l$  ($j=0,1,\dots,m$, $l=0,1,\dots,\mathfrak r$), thus one can put $ A_{j,l}:=A_j|_{\bm V_l}$. Combining this with 
\eqref{eq:oper:A}--\eqref{eq:linA} and $\mathfrak G$-equivariance of $\mathscr A$  implies  $\scrA(\mathscr E_{l,k})\subset \mathscr E_{l,k}$,  ($j=0,1,\dots,\mathfrak r$, $l=0,1,2,\dots,\mathfrak r$), thus one can put  $\scrA_{l,k}:=\scrA|_{\mathscr E_{l,k}}$.

   \vs 
To simplify our exposition, we replace assumption {\rm (A5)} by the following one: 
\medskip
\begin{itemize} 
\item[(A7)]   The $n\times n$-matrices $A_j$, $j=0,1,\dots,m$, satisfy  $A_{j,l}=\mu_j^l \id_{\bfV_l}$
for some number $\mu_j^l\in \br$ and $l=0,1,2,\dots, \mathfrak r$.
\end{itemize}

   \vs
\noindent   
Put  $\zeta_j:=e^{-i\tau_j}$, $j=1,2,\dots,m$.  Then, combining \eqref{eq:tauj}, {\rm (A0)}--{\rm (A7)},  \eqref{eq:oper:A}--\eqref{eq:linA} and \eqref{eq:action-L},  
one  obtains: 
\begin{equation}\label{eq:Ajk}
\mathscr A_{l,k}=\left[ 1+\frac 1{k^2+1} \left(\mu_0^l +\sum_{j=1}^m \mu_j^l\zeta_j^k -1\right)\right] \id_{\mathscr E_{l,k}}.
\end{equation}
Formula \eqref{eq:Ajk} implies that the spectrum of $\sigma(\mathscr A)$ consists of the  eigenvalues 
\begin{align}\label{eq-spectrum-mathcal-A}
\xi_{l,k}&:=1+\frac 1{k^2+1} \left(\mu_0^l +\sum_{j=1}^m \mu_j^l\zeta_j^k  -1\right) \notag\\
&= 1+\frac1{k^2+1} \left(\mu^l_0+ \sum_{j=1}^{\lfloor\frac {m}2\rfloor}2\mu_j^l\cos(k\tau_j)+\ve_m(-1)^k\mu^l_{\frac {m+1}2}-1 \right),
\end{align}
for $k=0,1,2,\dots$, $l=0,1,2,\dots, \mathfrak r$, where 
\begin{equation}\label{eq:formu-epsilon-m}
\ve_m=\begin{cases} 1 &\text{ if $m$ is odd}\\ 0 &\text{ if $m$ is even} \end{cases}.
\end{equation}

\begin{remark}\label{rem:non-equiv-multiplicity}
Notice that the (usual) multiplicity $m(\xi_{l,k})$ of $\xi_{l,k} \in \sigma(\mathscr A)$ is equal to 
\[
m(\xi_{l,k})=\begin{cases}
\text{dim\,} \bfV_l &\text{ for $k=0$},\\
2 \,\text{dim\,} \bfV_{l} &\text{ for $k>0$}.
\end{cases}
\]
\end{remark}
Since $\mathscr A$ is a {\it compact} linear field, one obtains:

\vs
\begin{corollary}\label{cor:non-singular-oper} 
Under the assumptions \eqref{eq:tauj}, {\rm (A0)}--{\rm (A7)},  the operator $\mathscr A$ is an isomorphism   if and only if  for every $k=0,1,2,\dots$ and $l=0,1,2,\dots, \mathfrak r$,  one has
\begin{equation}\label{eq:IsoA}   
-\mu_0-\sum_{j=1}^{\lfloor\frac {m}2\rfloor}2\mu_j^l\cos(k\tau_j)-\ve_m(-1)^k\mu^l_{\frac {m+1}2}\not=k^2.
\end{equation}
\end{corollary}

\vs

\subsection{Main symmetric result}\label{sec:app3:existence} Given \eqref{eq:V-decompo-Gamma},
put  
\begin{equation}\label{eq:ml-mult}
m_l:=\text{dim\,} \bfV_l/\text{dim\,} \cU_l,\quad  l=0,1,\dots, \mathfrak r,
\end{equation}
and for $k\ge 0$, 
\begin{equation}\label{eq:nu-l-k}
\nu(l,k):=
\begin{cases}
m_l &\text{if }\displaystyle \mu^l_0+\sum_{j=1}^{\lfloor\frac {m}2\rfloor}2\mu_j^l\cos(k\tau_j)+\ve_m(-1)^k\mu^l_{\frac {m+1}2} < -k^2,\\
0 &\text{ otherwise}.
\end{cases}
\end{equation}
Define
\begin{equation}\label{eq:m_l}
\mathfrak m_l:= \sum_{k=1}^\infty \nu(l,k), \quad l=0,1,\dots, \mathfrak r.
\end{equation}
Since $\mathscr A$ is a compact vector field,  $\nu(l,k)$ is different from zero only for finitely many  pairs of $(l,k)$, hence the integer $\mathfrak m_l$ is well-defined.
\vs
Given a $G$-representation $\bm W$, denote by $\mathfrak M_G(\bm W)$ the set of all maximal orbit types in $\bm W\setminus \{0\}$.
Under the assumptions   \eqref{eq:tauj}, {\rm (A0)}--{\rm (A7)} and \eqref {eq:IsoA}
(i.e. the operator $\mathscr A$ is an isomorphism),  put 
\begin{equation}\label{eq:mathcal-U-pm}
\cU^\pm:= \cU^\pm_0\oplus  \cU^\pm_1\oplus \dots \oplus  \cU^\pm_{\mathfrak r} 
\end{equation}
(cf. \eqref{eq:cU-l-pm}).
Then,  $\mathfrak M_G(\cU^-)\subset \mathfrak M_G(\mathcal U ^- \oplus \mathcal U ^+) =  \mathfrak M_G (\mathscr E)$.
Finally, for $(H)\in \mathfrak M_G(\cU^-)$, define (see \eqref{eq:m_l}):
\begin{equation}\label{eq:max-E} 
\mathfrak m(H):= \sum_{l=0}^{\mathfrak r}  \mathfrak m_l(H), \;\; \text{ where }\; \; \mathfrak m_l(H):= \begin{cases} \mathfrak m_l & \text{ if }\text{\rm coeff}^H(\deg_{\cU_l^-})\not=0,\\ 0 & \text {otherwise}. 
\end{cases}
\end{equation}
\vs 
We are now in a position to formulate our main symmetric result.
\vs 
\begin{theorem}\label{th:app3:main}
Let    $f:\br \times\bfV \times\bfV^m\times \bfV\to \bfV$ satisfy conditions  \eqref{eq:tauj}, {\rm (A0)}--{\rm (A7)} and \eqref {eq:IsoA}. Suppose that for some $(H) \in \mathfrak M_G(\cU^-)$, the number $\mathfrak m(H)$ is odd (cf. \eqref{eq:ml-mult}--\eqref{eq:max-E}).
Then, system \eqref{eq:delay-eq}  admits a non-constant periodic solution $u\in \mathscr E$  such that  $G_u=H$.
\end{theorem}
\begin{proof} For a subspace $K \subset \mathscr E$, denote by $B(K)$ the (open) unit ball in $K$.  Denote by $\sigma_-(\mathscr A)$ the negative spectrum of $\mathscr A$ and for $\lambda \in \sigma _-(\mathscr A)$, denote by $E(\lambda)$ the eigenspace of $\mathscr A$ corresponding to $\lambda$.  Then,  by \eqref{eq-spectrum-mathcal-A}, $\lambda:= \xi_{l,k}$ for some $k=0,1,2,\dots$ and $l=0,1,\dots,\mathfrak r$, where 
\[
\xi_{l,k}= 1+\frac1{k^2+1} \left( \mu^l_0+ \sum_{j=1}^{\lfloor\frac {m}2\rfloor}2\mu_j^l\cos(k\tau_j)+\ve_m(-1)^k\mu^l_{\frac {m+1}2}-1 \right)<0.
\]
By Proposition \ref{th:app3:abstract}, it is sufficient to show that coeff$^H(\omega)\not=0$. Combining the Multiplicativity property of the $G$-equivariant Brouwer degree  (see  Theorem \ref{thm:GpropDeg} in the Appendix) with formula 
\eqref{eq:prod-prop}, 
yields: 
\begin{align}\label{eq:simplific}
\omega&=(G)-\gdeg(\mathscr A,B(\mathscr E)) =(G)-\prod_{\lambda\in \sigma_-(\mathscr A)}\gdeg(-\id,B(E(\lambda)))\notag\\
&=(G) - \prod_{k=0}^\infty \prod_{l=0}^{\mathfrak r} (\gdeg(-\id,B(\mathscr E_{l,k}))^{\nu(l,k)} \notag\\ 
&= (G)-\prod_{k=0}^\infty \prod_{l=0}^{\mathfrak r} \left( \deg_{\cU_l^+} \right)^{\nu(l,k)} \cdot \prod_{k=1}^\infty \prod_{l=0}^{\mathfrak r} \left( \deg_{\cU_l^- }\right)^{\nu(l,k)}\notag\\
&=(G)-\prod_{k=0}^\infty \prod_{l=0}^{\mathfrak r} \left( \deg_{\cU_l^+} \right)^{\nu(l,k)} \cdot \prod_{(l,k)\in\mathfrak  I}  \left( \deg_{\cU_l^- }\right)^{\nu(l,k)},
\end{align}
where ``$\cdot$" stands for the product in the Burnside ring $A_0(G)$ and the (finite) set $\mathfrak I$ is given by
\[
\mathfrak I:=\Big\{(l,k): k = 1,2,\dots; \; l =  0,1,\dots, \mathfrak r; \; \nu(l,k)>0 \Big\}.
\]
Put 
\[
\Lambda_H:=\Big\{ (l,k)\in \mathfrak I: \text{coeff}^H(\deg_{\cU_l^-})\not=0  \Big\}, \quad \Lambda_H^c:=\mathfrak I \setminus \Lambda_H, 
\]
and 
\[
\mathfrak b:=\prod_{k=0}^\infty \prod_{l=0}^{\mathfrak r} \left( \deg_{\cU_l^+} \right)^{\nu(l,k)} \cdot \prod_{(l,k)\in\Lambda_H^c}  \left( \deg_{\cU_l^- }\right)^{\nu(l,k)}.
\]
Then, \eqref{eq:simplific} together with the Recurrence Formula (see Appendix, formulas \eqref{eq:RF-0} and \eqref{eq:bdeg-nL}) suggest:
\begin{equation}\label{eq-prelast}
\omega =(G)-\mathfrak b \cdot  \prod_{(l,k)\in\Lambda_H}  \left( \deg_{\cU_l^- }\right)^{\nu(l,k)} = 
(G)-\mathfrak b \cdot \prod_{p=1}^{\mathfrak m(H)} \big((G)-x_o(H)+\mathfrak c_p \big),
\end{equation}
where coeff$^H(\mathfrak b)=\text{coeff}^H(\mathfrak c_p)=0$ ($p=1,2,\dots , \mathfrak m(H)$), and
\begin{equation}\label{eq:coef-x-o-ireduc}
x_o=\begin{cases}
0 & \; \text{if $\text{coeff}^H(\deg_{\cU_{l}^-})=0$};\\ 
1 &  \; \text{if $\text{coeff}^H(\deg_{\cU_{l}^-})\not=0$ and }|W(H)|=2;\\
2 &    \; \text{if $\text{coeff}^H(\deg_{\cU_{l}^-})\not=0$ and } |W(H)|=1.
\end{cases}
\end{equation} 
Using \eqref{eq:coef-x-o-ireduc}, one can easily establish the following 
\begin{lemma}\label{lem:same-maximal} 
 Let  $(H)\in \mathfrak M_G(\cU^-)$ and $l,l^{\prime} \in \{0,1,...,\mathfrak r\}$. 
\begin{itemize}
\item[(a)] If 
$\text{\rm coeff}^H(\deg_{\cU_{l}^-})\not=0$ and $\text{\rm coeff}^H(\deg_{\cU_{l'}^-})\not=0$, then 
\[
\text{\rm coeff}^{H}(\deg_{\cU_{l}^-}\cdot \deg_{\cU_{l'}^-})=0.\]
\item[(b)] If 
$\text{\rm coeff}^H(\deg_{\cU_{l}^-})\not=0$ and $\text{\rm coeff}^H(\deg_{\cU_{l'}^-})=0$, then 
\[
\text{\rm coeff}^{H}(\deg_{\cU_{l}^-}\cdot \deg_{\cU_{l'}^-})=-x_o\] 
(in particular, different from zero).
\end{itemize}
\end{lemma}
Combining Lemma \ref{lem:same-maximal} with formula \eqref{eq-prelast} and the assumption that $\mathfrak m(H)$ is odd implies
\[
\text{coeff}^H\left[ -\mathfrak b \cdot\prod_{p =1}^{\mathfrak m(H)} \big((G)-x_o(H)+\mathfrak c_p \big)  \right]=-x_o.
\]
Therefore, 
\[
\text{coeff}^H\left( \omega)  \right) = x_o \not=0,
\]
and the conclusion follows from the Existence property of the Brouwer $G$-equivariant degree (see  Theorem \ref{thm:GpropDeg} in the Appendix). 
\end{proof}

\subsection{Non-symmetric result}
In the remaining part of this section, we assume that $\Gamma=\{e\}$, i.e. we are interested in the non-equivariant setting of problem \eqref{eq:delay-eq}. Also, we replace condition {\rm (A7)} by the following more general assumption:

\medskip
\vs
\begin{itemize}
\item[(A7)$^{\prime}$] ~~  For all $j$, $j^{\prime} \in \{0,1,2,\dots, m\}$ we have $A_jA_{j^{\prime}}=A_{j^{\prime}}A_j$.
\end{itemize}
\vs 
Under the assumption {\rm (A7)$^{\prime}$}, the matrices $A_j$, $j = 0,1,...,m$, share their generalized eigenspaces. To be more precise, there exist subspaces $V_l\subset \bfV$, $l=1,2,\dots,\mathfrak s$, such that:
\begin{itemize}
\item[(i)] $\bfV=V_1\oplus V_2\oplus \dots\oplus V_{\mathfrak s}$;
\item[(ii)] $\sigma(A_j)=\{\mu_j^1,\mu_j^2,\dots,\mu_j^{\mathfrak s}  \}$ for any $j = 0,1,...,m$;
\item[(iii)] the generalized eigenspace $E(\mu_j^l)$ of $\mu_j^l$ is exactly $V_l$  for any $j = 0,1,...,m$ and  $l=1,2,\dots,\mathfrak s$.
\end{itemize} 
Next, put $m_l:=\text{dim\,} V_l$ and define for any $k = 0,1,...$ 
\begin{equation}\label{eq:nu-k-l}
\nu(l,k):=
\begin{cases}
m_l &\text{if }\displaystyle \mu^l_0+ \sum_{j=1}^{\lfloor\frac {m}2\rfloor}2\mu_j^l\cos(k\tau_j)+\ve_m(-1)^k\mu^l_{\frac {m+1}2}<-k^2,\\
0 &\text{ otherwise}.
\end{cases}
\end{equation}
Finally, put 
\begin{equation}\label{eq:mathfrak-m-non}
\mathfrak m:=\sum_{k=1}^\infty \sum_{l=1}^{\mathfrak s}  \nu(k,l).
\end{equation}

\begin{corollary}\label{cor:app3:non-equi-main} 
Let    $f : \br \times\bfV \times\bfV^m\times \bfV\to \bfV$ satisfy conditions \eqref{eq:tauj}, {\rm (A0)}--{\rm (A5)}, {\rm (A7)$^{\prime}$} (cf. conditions
(i)--(iii) above).
Assume that for each $l=1,2,\dots,\mathfrak s$ and   $k = 0,1,...$
\begin{equation}\label{eq:IsoA-non-equi}
-\mu^l_0-\sum_{j=1}^{\lfloor\frac {m}2\rfloor}2\mu_j^l\cos(k\tau_j)-\ve_m(-1)^k\mu^l_{\frac {m+1}2}\not=k^2.
\end{equation}
Finally, suppose that the number $\mathfrak m$ is odd (see \eqref{eq:nu-k-l}--\eqref{eq:mathfrak-m-non}).
Then, system \eqref{eq:delay-eq}  admits a non-constant periodic solution $u\in \mathscr E$.
\end{corollary} 
\begin{proof} Apply the same arguments as in the proof of Theorem \ref{th:app3:main} with $H:=D_1^z$. 
\end{proof}
\vs

\section{Examples 
}\label{sec:app3.3}
\subsection{Examples of $f$   satisfying conditions (A0)--(A7)}
Consider the space $\bfV:=\br^n$  (with the norm {\it max}).
For the purpose of presenting an example of a map $f$ satisfying all the required assumptions,
consider a map $F : \mathbb R \times \bfV\times \bfV^m\times\bfV\to \bfV$ given by 
\begin{equation}\label{eq:form-F}
F(t,x,\bold y,z)=(p_1(t,x,\bold y,z),p_2(t,x,\bold y,z),\dots ,p_n(t,x,\bold y,z))^T\in \bfV,
\end{equation}
where 
\begin{equation}\label{eq:formula-p_i}
p_i(t,x,\bold y,z)=x_i^rQ_i(t,x,\bold y,z) + q_i(t,x,\bold y)+z_i ^2q_k(t,x,\bold y) \quad(i = 1,...,n) 
\end{equation}
satisfies the following conditions (which can be easily provided!):
\begin{itemize}
\item[(i)]  $r\ge 3$ is an odd integer;
\item[(ii)]  $Q_i(t,x,\bold y,z$) is continuous, $2\pi$-periodic in $t$-variable and 
\begin{equation*}
\overline{Q}: = \inf \{Q_i(t, x,\bold y,z) \,: i=1,\dots,n; \,  t \in \mathbb R; \, x,z \in \bfV; \, \bold y \in \bfV^m\} > 0;
\end{equation*}
\item[(iii)] for any fixed $x$, $\bold y$ and $t$,  the function $|Q_i(t, x,\bold y,z)|$ is bounded with respect to $z$; 
\item[(iv)]  $q_i(t,x,\bold y)$ is $2\pi$-periodic in $t$-variable and for any fixed $t \in \mathbb R$, the map $q_i(t,x,\bold y)$ is a homogeneous polynomial of degree $d$ such that $r > d >1$;
\item[(v)] $x_i \cdot q_i(t,x,\bold y)\ge 0$ for all   $t \in \mathbb R, \, x  \in \bfV, \, \bold y \in \bfV^m$.
\end{itemize}
Take matrices 
$A_j:\bfV \to\bfV$, $j=0,1,\dots,m$ (to be specified later on), and consider the map 
\begin{equation}\label{eq:the-map-f}
f(t,x,\bold y,z):=A_0x+\sum_{j=1}^m A_j y^j +F(t,x,\bold y,z), \quad x,z\in \bfV, \; \bold y = (y^1,...,y^m)^T \in \bfV^m,\; t\in \br.
\end{equation}
Clearly, if $F$ is given by  \eqref{eq:form-F} and satisfies  (ii) and (iv), then $f$  satisfies {\rm (A0)}.  Let us show that if $F$ is given by \eqref{eq:form-F} and satisfies (i)--(v), then $f$ satisfies the Hartman-Nagumo conditions {\rm (A1)}--{\rm (A3)}. 
In fact, if $|x|\ge |\bold y|$, then  (ii) and (v) imply:
%
\begin{align}\label{eq:prelim-est}
x\bullet f(t,x,\bold y, z)&=\sum_{i=1}^n\Big({x_i^{r+1}Q_i(t,x,\bold y,z)}+{x_iq_i(t,x,\bold y)}+z_i^2x_iq_i(t,x,\bold y)\Big) \notag\\
&+x\bullet A_0 x+x\bullet \sum_{j=1}^m A_j y^j \notag
\\
&\ge\sum_{i=1}^n \overline{Q} x_i^{r+1} - \sum_{i=1}^n x_iq_i(t,x,\bold y) -|A_0||x|^2-\sum_{j=1}^m |A_j|\,| y^j||x| \notag\\
&\ge \overline{Q} |x|^{r+1} - \sum_{i=1}^n x_iq_i(t, x,\bold y) - \left(|A_0|+\sum_{j=1}^m |A_j|\right)\,|x|^2 
\end{align}
By (iv), for any $t \in \mathbb R$, the powers in monomials of $q_i(t,x,\bold y$ satisfy 
\[
d:=\sum_{j=0}^m\sum_{i=1}^n\alpha^j_i < r.
\]
Hence (recall, we assume $|x|\ge |\bold y|$), 
\begin{align}\label{eq:prelim-est1}
\left|x_k\prod_{i=1}^nx_i^{\alpha^0_i}\cdot\prod_{j=1}^m\prod_{i=1}^n (y^j_i)^{\alpha_l^j}\right|
	\le |x|^{d+1}\le (1+|x|)^r.
\end{align}
Combining \eqref{eq:prelim-est} and \eqref{eq:prelim-est1} implies that there exists a constant $a>0$ such that 
\begin{align*}
x\bullet f(t,x,\bold y,z)&\ge \overline{Q} |x|^{r+1}-a(1+|x|^r)-\left(|A_0|+\sum_{j=1}^m |A_j|\right)\,|x|^2 \quad \text{for any $x \in \bfV$},
\end{align*}
which implies (see (i)) that for a sufficiently large $R > 0$, if $|x|>R$, then $x\bullet f(t,x,\bold y, z)>0$, so condition {\rm (A1)} is satisfied. 

\medskip

Put 
{\small \begin{align*}
C&:=\max\left\{  |x_i^rQ_i(t,x,\bold y,z)|+|q_i(t,x,\bold y)|+\sum_{j=0}^m\|A_j\|R  \, : \,  i=1,\dots,n; \; |x|\le R; \; |\bold y|\le R; \, t\in \br\right\}, \\
D&:=\max\{ |q_i(t,x,\bold y)| \, : \, i=1,\dots, n; \; |x|\le R; \; |\bold y|\le R; \,  t\in \br\},
\end{align*}}

\noindent
and observe that by (iii), $C$ is correctly defined. Clearly,
condition (A2) is satisfied with $\phi(s)=C+D|s|^2$, $s\in \br$. 

\medskip
Next, put 
{\small \[
K:=\min\left\{ \sum_{i=1}^n\Big(x_i^{r+1}Q_i(t,x,\bold y,z)+x_iq_i(t,x,\bold y)+x\bullet A_0 x+x\bullet \sum_{j=1}^m A_j y^j\Big)  \, : \, |x|\le R, \; |\bold y|\le R,\; t\in \br  \right\}
\]}
(again, $K$ is correctly defined by (iii)). Then, for all $|x|\le R$, $|\bold y|\le R$, $t\in \br$, one has:
\begin{align*}
|f(t,x,\bold y,z)|&\le x\bullet f(t,x,\bold y,z)+K-D|z|^2\\
&\le (D+1)(x\bullet f(t,x,\bold y,z)+|z^2|) + K
\end{align*}
so that condition {\rm (A3)} is satisfied.

It follows immediately from \eqref{eq:form-F}--\eqref{eq:the-map-f} and (i) that condition (A5) is  satisfied. Notice that it is possible to choose the related polynomials in such a way  that condition {\rm (A4)} is also satisfied.  

Assume that $\bfV=\br^n$ is an orthogonal $\Gamma$-representation where $\Gamma\le S_n$ acts on the vectors in $\br^n$ by permuting their coordinates. One can easily assure, by identifying the $\Gamma$-symmetric interactions between the coordinates in $\bfV$, that condition {\rm }(A6) is satisfied. To fulfill condition {\rm (A7)}, it is enough to
suppose that the $\Gamma$-isotypic decomposition of $V$ is of the form 
\[
\bfV=\bfV_0\oplus \bfV_1\oplus \dots\oplus \bfV_r,
\]
where $\bfV_l\simeq \cU_l$ is an irreducible $\Gamma$-representation of real type.
Then, for each of the $\Gamma$-equivariant matrices $A_j$, one necessarily has $A_j|_{\bfV_l}=\mu_j^l \id_{\bfV_l}$, so that  condition {\rm (A7)} is satisfied. 
\vs
\subsection{Example of system  \eqref{eq:delay-eq} with dihedral symmetries}\label{subsec-Didedral-example} For the sake of simplicity, 
in what follows, we will assume that the dimension $n$  of the space $\bfV$ is {\it odd}. 
\vs
Let us consider, as a particular case of the group $\Gamma$, the dihedral group $D_n\le S_n$, where the rotation $\gamma:=e^{\frac{2\pi i}{n}} $ is identified with the permutation $(1,2,3,\dots,n)$ and the reflection $\kappa$, $\kappa z=\overline z$, with $(2,n)(3,n-1)\dots$. We will also assume that the $D_n$-equivariant matrices  $A_j$ are given by
\begin{equation}\label{eq:Aj}
A_j:=\begin{bmatrix}
   ~ a_j~~&~~b_j~~&0&~~\dots~~&~~0~~&~~b_j~\\
    b_j&a_j&~b_j~&\dots&0&0\\
    0&b_j&a_j&\dots&0&0\\
  \vdots&\vdots&\vdots&\ddots&\vdots&\vdots\\
   0&0&0&\vdots&a_j&b_j\\
    b_j&0&0&\dots&b_j&a_j
    \end{bmatrix} , \quad j=0,1,2,\dots,m.
\end{equation} 
Then, the $D_n$-isotypic decomposition of $\bfV$ is given by
\[
\bfV = \bfV_0\oplus\bfV_1\oplus \bfV_\oplus \dots\oplus \bfV_{\mathfrak r}, \quad \mathfrak r=\left\lfloor \tfrac n2\right\rfloor,
\]
where $\bfV_0 \simeq \mathcal U_0$  is the trivial one-dimensional $D_n$-representation and 
$\bfV_l \simeq \mathcal U_l$, $l = 1, ...,\mathfrak r$,  is the irreducible two-dimensional $D_n$-representation, where $\gamma$ acts on $\mathbb R^2 \simeq \mathbb C$ by usual complex multiplication by $\gamma^l$. Since $\bfV_l$, $l = 0,1,...,\mathfrak r$, are irreducible, it follows that  
\[
A_{j,l}=\mu_j^l\id_{\bfV_l}, \quad \text{ where }\; \mu_j^l=a_j+2b_j\cos \frac{2\pi l}{n}, \;\; l=0,1,\dots, \mathfrak r; \;j = 0,1,...,m.
\]
Clearly, the orbit type $(D_1^z\times D_n)$ is maximal in  $\mathscr E\setminus \{0\}$.  Suppose $n=p_1^{\ve_1}p_2^{\ve_2}\dots p_{\mathfrak k}^{\ve_{\mathfrak k}}$, where $\ve_s>0$,  and $p_s$, $s=1,2,\dots,\mathfrak k$, are the prime numbers such that 
	$2< p_1<p_2<\dots < p_{\mathfrak k}$. Then, for $s\in \{1,2,\dots, \mathfrak k\}$, put 
\begin{equation}\label{eq:-n-s}
n_s:=\frac n{p_s}\quad\quad   \text{and}  \quad\quad  H_s:= (D_1\times \bz_2){^{D_1^z}}\times_{\bz_2}^{\bz_{n_s}}D_{n_s}
\end{equation} 
(see Appendix, formulas \eqref{eq:amalgam-projections}--\eqref{eq:amalg}, for the amalgamated notation used here). One can easily see that 
for each $s = 1,2,\dots, \mathfrak k$,
the orbit type $(H_s)$ is maximal in $\mathscr E\setminus \{0\}$ and $(H_s)\in \mathfrak M_G(\cU^-)$.
Notice that for $l=0,1,2,\dots,\mathfrak r$, one has $m_l=1$ (see \eqref{eq:ml-mult}). Therefore, we have the following immediate consequence of Theorem \ref{th:app3:main}.
\vs
\begin{corollary}\label{cor-Dihedral} 
Assume that $\Gamma=D_n\le S_n$ (here $n$ is an odd number) acts on $\bfV:=\br^n$ by permuting coordinates of vectors and suppose that  $n=p_1^{\ve_1}p_2^{\ve_2}\dots p_{\mathfrak k}^{\ve_{\mathfrak k}}$, where $\ve_s>0$,  and $p_s$, $s=1,2,\dots,\mathfrak k$, are prime numbers such that 
	$2< p_1<p_2<\dots < p_{\mathfrak k}$. 
Let    $f:\br \times\bfV \times\bfV^m\times \bfV\to \bfV$ satisfy conditions  \eqref{eq:tauj}, {\rm (A0)}--{\rm (A7)} and \eqref {eq:IsoA}.
Assume, finally, that for some  $s\in \{1,2,\dots, \mathfrak k\}$, the number $\mathfrak m(H_s)$ is odd  (cf. \eqref{eq:-n-s} and \eqref{eq:max-E}).
Then, system \eqref{eq:delay-eq}  admits a non-constant periodic solution $u\in \mathscr E$  such that  $G_u=H_s$.
\end{corollary}

\vs
\noi{\bf Example: case $m=4$ and $n=3$:} We  choose $a_0=-1$, $b_0=-2$, $a_1=a_4=-2$, $b_1=b_4=-4$, $a_2=a_3=-3$, $b_2=b_3=-5$. We will also assume that $\tau_j:= \frac{2\pi j}{5}$, $j=0,1,2,3,4$ (see condition  \eqref{eq:tauj}).  In this case, 
\begin{equation}\label{eq:gather}
\mu_0^0= -5, \quad \mu_1^0= -10, \quad \mu_2^0=-13, \quad \mu_3^0=-13, \quad \mu_4^0=-10
\end{equation}
\begin{equation}\label{eq:gather1}
\mu_0^1= 1, \quad \mu_1^1=2 , \quad \mu_2^1=2, \quad \mu_3^1=2, \quad \mu_4^1=2.
\end{equation}
Notice that 
\[
\cos (\tau_1 k)=\begin{cases}
\frac{\sqrt 5-1}{4} &\text{ if }   k = 5p \pm 1; \\
\frac{-1-\sqrt 5}{4} & \text{ if } k= 5p \pm 3;  \\
1 &\text{ if } k=5p
\end{cases} 
\]
$(p = 0,1,... )$,
and the eigenvalues of the operator $\mathscr A$ are given by
\begin{align*}
\xi_{l,k} &=1+\frac { \mu_0^l + 2\mu_1^l\cos (\tau_1 k)+ 2\mu_2^l\cos(\tau_1 2k)+2\mu_2^l\cos(\tau_1 3k) +2\mu_1^l\cos(\tau_1 4k)-1}{k^2+1}\\
&=1+\frac { \mu_0^l + 4\mu_1^l\cos (\tau_1 k)+ 4\mu_2^l\cos(\tau_1 2k)-1}{k^2+1}\label{eq-nunu}
\end{align*}
(see \eqref{eq:gather}--\eqref{eq:gather1}).
Hence,
\[
\xi_{0,k}=\begin{cases}
1-\frac{98}{k^2+1} &\text{ if } k=5l,\\
1+\frac{17+3\sqrt 5}{k^2+1} &\text{ if } k=5l+1, 5l+2,\\
1+\frac{17-3\sqrt 5} {k^2+1} & \text{ if } k=5l+3,5l+4,
\end{cases}
\quad   \xi_{1,k}=\begin{cases}
1+\frac{16}{k^2+1} &\text{ if } k=5l,\\
1-\frac{16}{k^2+1} &\text{ if } k=5l+1, 5l+2,\\
1-\frac{16\sqrt 5} {k^2+1} & \text{ if } k=5l+3,5l+4.
\end{cases}
\]
Consequently, the negative spectrum of the operator $\mathscr A$ is:
\[
\sigma_-(\mathscr A)=\left\{ \xi_{0,0}=-97,\; \xi_{1,1}=-7, \; \xi_{1,2}=-\frac{11}5, \; \xi_{1,3}=-\frac{3}{5}  \right\}.
\]
Also, $G:=D_1\times \bz_2\times D_3$ and the related basic degrees are:
\begin{align*}
\deg_{\cU_0^-}&=(G) -(D_1^z\times D_3),\\
\deg_{\cU_0^+}&=(G) -(D_1\times D_3),\\
\deg_{\cU_1^-}&=(G) -(D_1^z\times D_1)-((D_1\times\bz_2)^{D_1^z}\times_{\bz_2}  ^{\{e\}}D_1)+(D_1^z\times \bz_1)\\   
\deg_{\cU_1^+}&=(G) -(D_1\times D_1)-((D_1\times\bz_2)^{D_1}\times_{\bz_2} ^{\{e\}} D_1)+(D_1^z\times \bz_1),
\end{align*}
where $\{e\}$ stands for the unit subgroup in $D_3$.
The maximal orbit types in $\mathscr E\setminus \{0\}$ are: 
\[
\mathfrak M_G(\mathscr E)=\left\{  (D_1^z\times D_3), \;   (D_1^z\times D_3),\;  ((D_1\times\bz_2)^{D_1^z}\times_{\bz_2} ^{\{e\}} D_1),\; ((D_1\times\bz_2)^{D_1}\times_{\bz_2} ^{\{e\}} D_1)  \right\}
\]
and 
\[
\mathfrak M_G(\cU^-)=\left\{  (D_1^z\times D_3), \;  ((D_1\times\bz_2)^{D_1^z}\times_{\bz_2} ^{\{e\}} D_1)\right\}.
\]
On the other hand, 
\[
\gdeg(\mathscr A,B_1(0))=\deg_{\cU_0^+}\cdot \deg_{\cU_1^+}^3\cdot \deg_{\cU_1^-}^3=\deg_{\cU_0^+}\cdot \deg_{\cU_1^+}\cdot \deg_{\cU_1^-},
\]
which implies that  for $H:=(D_1\times \bz_2)^{D_1^z}\times_{\bz_2} ^{\{e\}} D_1$, one has $\mathfrak m(H) = 1$. Hence, 
\[
\text{coeff}^H(\omega) = \text{coeff}^H\big((G)-\deg_{\cU_0^+}\cdot \deg_{\cU_1^+}\cdot \deg_{\cU_1^-}\big)=1.
\]
Therefore, there exists an orbit of non-constant periodic solutions to  system \eqref{eq:delay-eq} with the orbit type exactly $(H)$. 

\vs

To double check the obtained result, one can also  use the GAP package {\tt EquiDeg}, as it is presented below:

\vs

\noi{\small {\bf GAP Code:} $G:=D_1\times \bz_2\times  D_3$. 

\begin{lstlisting}[language=GAP, frame=single]
LoadPackage( "EquiDeg" );
gr1 := SymmetricGroup( 2 );
# create the product of D_1 and Z_2
gr2:= DirectProduct( gr1, gr1 );
# create group G
gr3 := pDihedralGroup( 3 ); 
G := DirectProduct( gr2, gr3 );
# create and name CCSs of gr2 and gr3
ccs_gr2:= ConjugacyClassesSubgroups( gr2 );
ccs_gr2_names := [ "Z1", "Z1p", "D1",  "D1z", "D1p"];
ccs_gr3:=ConjugacyClassesSubgroups( gr3 );
ccs_gr3_names:=["Z1", "D1", "Z3","D3"];
SetCCSsAbbrv(gr2, ccs_gr2_names); 
SetCCSsAbbrv(gr3, ccs_gr3_names);
ccs := ConjugacyClassesSubgroups( G );
# create characters of irreducible G-representations
irr := Irr( G );
# compute the corresponding to irr[k[] basic degree 
deg0m := BasicDegree( irr[6] );
deg0p := BasicDegree( irr[7] );
deg1m := BasicDegree( irr[9] ); 
deg1p := BasicDegree( irr[10] ); 
# obtaining amalgamation symbols
Print( AmalgamationSymbol(ccs[16]));
# maximal orbit types in E
max:=MaximalOrbitTypes(irr[6]+irr[7]+irr[9]+irr[10]); 
# unit element in Burnside ring AG
u := -BasicDegree( irr[1] );
# compute \gdeg  of F on Omega
deg:= u-deg0p*deg1p*deg1m;
\end{lstlisting}
}

\vs

\appendix
\section{Equivariant Brouwer Degree Background}
\label{subsec:G-degree}

\noi
{\bf (a) Amalgamated Notation.} 
Given two groups $G_{1}$ and
$G_{2}$, 
the well-known result of \'E. Goursat (see \cite{Goursat}) provides the following description of a
subgroup $\mathscr U \leq G_{1}\times G_{2}$: 
there exist subgroups
$H\leq G_{1}$ and $K\leq G_{2}$, a group $L$, and two epimorphisms
$\varphi:H\rightarrow L$ and $\psi:K\rightarrow L$ such that
\begin{equation*}
\mathscr U =\{(h,k)\in H\times K:\varphi(h)=\psi(k)\}.
\end{equation*}
The widely used notation for $\mathscr U$ is 
\begin{equation}\label{eq:amalgam-projections}
\mathscr U:=H\prescript{\varphi}{}\times_{L}^{\psi}K,
\end{equation}
in which case $H\prescript{\varphi}{}\times_{L}^{\psi}K$ is called an
\textit{amalgamated} subgroup of $G_{1}\times G_{2}$.

In this paper, we are interested in describing conjugacy classes of $\mathscr U$. Therefore, to make notation \eqref{eq:amalgam-projections} simpler and
self-contained, it is enough to indicate $L$,  
$Z=\text{Ker\thinspace}(\varphi)$ and 
$R=\text{Ker\thinspace}(\psi)$. Hence, instead of  
\eqref{eq:amalgam-projections}, we use the following notation:
\begin{equation}
\mathscr U=:H{\prescript{Z}{}\times_{L}^{R}}K~ \label{eq:amalg}.
\end{equation}

\vs\noi
{\bf (b) Equivariant Notation.} Below $G$ stands for a compact Lie group.
For a subgroup $H$ of $G$, 
denote by $N(H)$ the
normalizer of $H$ in $ G$ and by $W(H)=N(H)/H$ the Weyl group of $H$.  The symbol $(H)$ stands for the conjugacy class of $H$ in $G$. 
Put $\Phi(G):=\{(H): H\le G\}$.
The set $\Phi (G)$ has a natural partial order defined by 
$(H)\leq (K)$ iff $\exists g\in G\;\;gHg^{-1}\leq K$. 
Put $\Phi_0 (G):= \{ (H) \in \Phi(G) \; : \; \text{$W(H)$  is finite}\}$.

For a $G$-space $X$ and $x\in X$, denote by
$\mathcal G_{x} :=\{g\in G:gx=x\}$  the {\it isotropy group}  of $x$
and call $(G_{x})$   the {\it orbit type} of $x\in X$. Put $\Phi(G,X) := \{(H) \in \Phi_0(G) \; : \; 
(H) = (G_x) \; \text{for some $x \in X$}\}$ and  $\Phi_0(G,X):= \Phi(G,X) \cap \Phi_0(G)$. For a subgroup $H\leq G$, the subspace $
X^{H} :=\{x\in X:G_{x}\geq H\}$ is called the {\it $H$-fixed-point subspace} of $X$. If $Y$ is another $G$-space, then a continuous map $f : X \to Y$ is called {\it equivariant} if $f(gx) = gf(x)$ for each $x \in X$ and $g \in G$. 
Let $V$ be a finite-dimensional  $G$-representation (without loss of generality, orthogonal).
Then, $V$  decomposes into a direct sum 
\begin{equation}
V=V_{0}\oplus V_{1}\oplus \dots \oplus V_{r},  \label{eq:Giso}
\end{equation}
where each component $V_{i}$ is {\it modeled} on the
irreducible $G$-representation $\mathcal{V}_{i}$, $i=0,1,2,\dots ,r$, that is, $V_{i}$  contains all the irreducible subrepresentations of $V$
equivalent to $\mathcal{V}_{i}$. Decomposition  \eqref{eq:Giso}  is called  $G$\textit{-isotypic  decomposition of} $V$.
\vs \noi
{\bf (b) Axioms of Equivariant Brouwer Degree.} Denote by  $\mathcal{M}^{G}$ the set of all admissible $G$-pairs and let $A_0(G)$ stand for the Burnside ring of $G$ (see Introduction, items (a) and (b) respectively). The following result (cf.  \cite{AED}) can be considered as an axiomatic definition of the {\it $G$-equivariant Brouwer degree}.

\begin{theorem}
\label{thm:GpropDeg} There exists a unique map $\gdeg :\mathcal{M}
^{G}\to A_0(G)$, which assigns to every admissible $G$-pair $(f,\Omega)$ an
element $\gdeg(f,\Omega)\in A_0(G)$
\begin{equation}
\label{eq:G-deg0}\gdeg (f,\Omega)=\sum_{(H)}%
{n_{H}(H)}= n_{H_{1}}(H_{1})+\dots+n_{H_{m}}(H_{m}),
\end{equation}
satisfying the following properties:

\begin{itemize}
\item[] \textbf{(Existence)} If $\gdeg (f,\Omega)\ne
0$, i.e., $n_{H_{i}}\neq0$ for some $i$ in \eqref{eq:G-deg0}, then there
exists $x\in\Omega$ such that $f(x)=0$ and $(G_{x})\geq(H_{i})$.

\item[] \textbf{(Additivity)} Let $\Omega_{1}$ and $\Omega_{2}$
be two disjoint open $G$-invariant subsets of $\Omega$ such that
$f^{-1}(0)\cap\Omega\subset\Omega_{1}\cup\Omega_{2}$. Then,
\begin{align*}
\gdeg (f,\Omega)=\gdeg (f,\Omega_{1})+\gdeg 
(f,\Omega_{2}).
\end{align*}

\item[] \textbf{(Homotopy)} If $h:[0,1]\times V\to V$ is an
$\Omega$-admissible $G$-homotopy, then
\begin{align*}
\gdeg (h_{t},\Omega)=\mathrm{constant}.
\end{align*}

\item[] \textbf{(Normalization)} Let $\Omega$ be a $G$-invariant
open bounded neighborhood of $0$ in $V$. Then,
\begin{align*}
\gdeg (\id,\Omega)=(G).
\end{align*}

\item[] \textbf{(Multiplicativity)} For any $(f_{1},\Omega
_{1}),(f_{2},\Omega_{2})\in\mathcal{M} ^{G}$,
\begin{align*}
\gdeg (f_{1}\times f_{2},\Omega_{1}\times\Omega_{2})=
\gdeg (f_{1},\Omega_{1})\cdot \gdeg (f_{2},\Omega_{2}),
\end{align*}
where the multiplication `$\cdot$' is taken in the Burnside ring $A_0(G )$.

\item[] \textbf{(Recurrence Formula)} For an admissible $G$-pair
$(f,\Omega)$, the $G$-degree \eqref{eq:G-deg0} can be computed using the
following Recurrence Formula:
\begin{equation}
\label{eq:RF-0}n_{H}=\frac{\deg(f^{H},\Omega^{H})- \sum_{(K)>(H)}{n_{K}\,
n(H,K)\, \left|  W(K)\right|  }}{\left|  W(H)\right|  },
\end{equation}
where $\left|  X\right|  $ stands for the number of elements in the set $X$
and $\deg(f^{H},\Omega^{H})$ is the Brouwer degree of the map $f^{H}%
:=f|_{V^{H}}$ on the set $\Omega^{H}\subset V^{H}$.
\end{itemize}
\end{theorem}

The $\gdeg(f,\Omega)$ is 
 called the {\it $G$%
-equivariant  Brouwer degree of $f$ in $\Omega$}.


\vs\noi
{\bf (c) Computation of Brouwer Equivariant Degree.} 
Consider a $G$-equivariant linear isomorphism $T:V\to V$ and assume that $V$
has a $G$-isotypic  decomposition \eqref{eq:Giso}. Then, by the
Multiplicativity property,
\begin{equation}\label{eq:prod-prop}
\gdeg (T,B(V))=\prod_{i=0}^{r}\gdeg 
(T_{i},B(V_{i}))= \prod_{i=0}^{r}\prod_{\mu\in\sigma_{-}(T)} \left(
\deg_{\mathcal{V} _{i}}\right)  ^{m_{i}(\mu)}%
\end{equation}
where $T_{i}=T|_{V_{i}}$, $\sigma_{-}(T)$ denotes the real negative
spectrum of $T$ (i.e., $\sigma_{-}(T)=\left\{  \mu\in\sigma(T):\mu<0\right\})$ and $m_i(\mu) = \dim\big(E(\mu) \cap V_i\big)$
(here $E(\mu)$ stands for the generalized eigenspace of $T$ corresponding to $\mu$). 
Notice that the basic degrees can be effectively computed from \eqref{eq:RF-0}: 
\begin{align*}
\deg_{\mathcal{V} _{i}}=\sum_{(H)}n_{H}(H),
\end{align*}
where 
\begin{equation}\label{eq:bdeg-nL}
n_{H}=\frac{(-1)^{\dim\mathcal{V} _{i}^{H}}- \sum
_{H<K}{n_{K}\, n(H,K)\, \left|  W(K)\right|  }}{\left|  W(H)\right|}.
\end{equation}

\end{document}